\documentclass[11pt]{amsart}
\usepackage[T1]{fontenc}
\usepackage[utf8]{inputenc}
\usepackage{lmodern,amsmath,amssymb,mathtools,microtype}
\usepackage[a4paper,textwidth=155mm,textheight=235mm,centering]{geometry}
\usepackage[hidelinks,unicode,hypertexnames=false]{hyperref}
\hypersetup{pdftitle={Metric Poissonian pair correlation for n log n},pdfauthor={}}
\numberwithin{equation}{section}
\newtheorem{theorem}{Theorem}[section]
\newtheorem{proposition}[theorem]{Proposition}
\newtheorem{lemma}[theorem]{Lemma}
\newtheorem{corollary}[theorem]{Corollary}
\newcommand{\ps}{\bar\Psi}
\newcommand{\ph}{\bar\Phi_\ell}
\newcommand{\ee}{\mathrm e}
\newcommand{\dd}{\mathrm d}
\newcommand{\RR}{\mathbb R}\newcommand{\ZZ}{\mathbb Z}
\newcommand{\supp}{\operatorname{supp}}\newcommand{\diag}{\operatorname{diag}}
\newcommand{\sgn}{\operatorname{sgn}}
\allowdisplaybreaks[2]
\title[Metric pair correlation for $n\log n$]{Metric Poissonian Pair Correlation Property\\ Part 2: The Case \(n\log n\)} 
\author{Bo Tan, \ \  Zhen Zhang, \ \ Qing-Long Zhou}
\date{}
\begin{document}
\raggedbottom
\begin{abstract}
We prove that $(\alpha n\log n)_{n\ge1}$ has Poissonian pair correlation for almost every real $\alpha$, solving a Problem  posed by Aistleitner, El-Baz and Munsch [Open Problem 3, Geom. Funct. Anal. 2021]. Our main contribution is a new four-point counting estimate with a fixed power saving over $N^{7/3}$ at tolerance $N^{-1}$, valid on the ranges left after a first-moment and Borel--Cantelli argument. Its proof reduces the counting condition to the required proximity of an implicitly defined real root to an integer; a Fej\'er kernel majorizes this condition, and multidimensional van der Corput $A$- and $B$-processes supply the saving. A localized Mellin argument then yields the second-moment estimate needed for the metric conclusion.
\end{abstract}
\maketitle

\section{Introduction}\label{sec:introduction}

This is the second of two papers addressing the open problems
on metric Poissonian pair correlation for the sequences
$n+\log n$ and $n\log n$ posed in
Aistleitner, El-Baz and Munsch~\cite[Section~9, Open Problems~2 and~3]{AEM}.
Our first paper~\cite{PaperI} establishes the metric Poissonian
property for $n+\log n$ and discusses the relation with uniform
distribution and the background of both problems.
The present paper completes the treatment of these two sequences
by resolving the $n\log n$ case.

Recall that a real sequence $(x_n)_{n\ge1}$ has
\emph{Poissonian pair correlation} (PPC) if
\[
 R_N(s;x):=\frac1N\#\{1\le m\ne n\le N:
       \|x_m-x_n\|_{\RR/\ZZ}\le s/N\}\longrightarrow 2s
 \qquad (N\to\infty)
\]
for every $s\ge0$, where $\|t\|_{\RR/\ZZ}$ denotes the distance
to the nearest integer. A real sequence $(a_n)_{n\ge1}$ is
\emph{metric Poissonian} if $(\alpha a_n)_{n\ge1}$ has PPC
for Lebesgue almost every $\alpha\in\RR$. 

The main theorem of this paper establishes the metric Poissonian pair correlation property for the \(n\log n\) case, thereby giving an affirmative answer to the related open problem mentioned above.

\begin{theorem}\label{thm:main}
For Lebesgue almost every $\alpha\in\RR$, simultaneously for every
$s\ge0$,
\[
 \frac1N\#\{1\le m\ne n\le N:
    \|\alpha(m\log m-n\log n)\|_{\RR/\ZZ}\le s/N\}
 \longrightarrow 2s.
\]
\end{theorem}

We first apply a first-moment estimate and the Borel--Cantelli lemma to show that the contribution of pairs $(n,n+h)$ with $n\le N/(\log N)^4$ or $h\le N/(\log N)^4$ tends to zero for almost every $\alpha$ along a suitable subsequence. For the remaining pairs, a localized Mellin argument bounds the second moment in terms of four-point counts on dyadic ranges, counting pairs of differences of $(n\log n)_{n\ge1}$ whose absolute difference is at most $N^{-1}$. A general benchmark for these counts is provided by Bloom's bound for the additive energy of convex sets~\cite[Theorem~2]{Bloom}, which, after rescaling and discretization, yields an upper bound $O_\varepsilon(N^{123/50+\varepsilon})$ for each such count. In our framework, the second-moment reduction gives, up to logarithmic factors, $N^{-7/3}$ times the largest of these counts; see \eqref{eq:metric-moment}. To obtain power decay through this reduction, we therefore require a sharper counting bound with a fixed exponent strictly below $7/3$.

The main contribution of this paper is a bound of the form $O(N^{7/3-\sigma})$, for some absolute $\sigma>0$, for these four-point counts, uniformly over the dyadic ranges remaining after the preliminary removal; see Proposition~\ref{prop:energy}. This supplies the power saving required by the second-moment argument. The proof exploits the specific structure of the differences of $n\log n$. We use the necessary condition that an implicitly defined real root lie close to an integer, and majorize this condition by a Fej\'er kernel. The resulting exponential sums are estimated by multidimensional van der Corput $A$- and $B$-processes, together with a third-derivative estimate, to obtain the stated counting bound. A further technical ingredient is the choice of a common subsequence along which both the first-moment and second-moment estimates are summable and whose consecutive terms have ratio tending to one. This allows the preliminary removal and the counting estimate to be combined: we restore the discarded pairs along this subsequence and then use monotonicity to obtain the pair correlation limit for all $N$.

Section~\ref{sec:metric} treats smoothing, small pairs and the
Fourier expansion. Section~\ref{sec:moment-proof} gives the
localized Mellin reduction and deduces
Theorem~\ref{thm:main} from the four-point bound.
Sections~\ref{sec:first}--\ref{sec:remaining-shifts} prove that bound.
The appendices contain the root identities, derivative calculations,
and stationary-phase estimates.

Write $L=\log N$, let $f(x)=x\log x$ for $x>0$, and set
$G_h(n)=f(n+h)-f(n)$.
The notation $n\asymp P$ means that $n/P$ lies in a fixed compact
subinterval of $(0,\infty)$; other ranges are interpreted in the same way.
We write $\ee(t)=\exp(2\pi it)$ and, for $t\in\mathbb R$,
$\|t\|=\|t\|_{\mathbb R/\mathbb Z}=\min_{n\in\mathbb Z}|t-n|$.
For a scalar, $|\cdot|$ denotes absolute value or complex modulus.
For $u=(u_1,\ldots,u_m)\in\mathbb R^m$, put
\[
 |u|=\left(\sum_{i=1}^m u_i^2\right)^{1/2},\qquad
 |u|_\infty=\max_{1\le i\le m}|u_i|.
\]
For a matrix $A$, use the induced operator norm
\[
 \|A\|=\sup_{|u|=1}|Au|.
\]
For an $r$-linear form $T$, set
$\|T\|=\sup_{|u_1|,\ldots,|u_r|\le1}|T[u_1,\ldots,u_r]|$.
The notation $\partial^rG(x)$ denotes the $r$th total derivative,
viewed as a symmetric $r$-linear form when $G$ is scalar-valued.
For a multi-index $\alpha$, $|\alpha|=\sum_i\alpha_i$.
For a smooth scalar- or vector-valued function $G$ on a set
$U\subset\mathbb R^m$, put
\[
 \|G\|_{L^\infty(U)}=\sup_{u\in U}|G(u)|,\qquad
 \|G\|_{C^r(U)}=
 \max_{|\alpha|\le r}\|\partial^\alpha G\|_{L^\infty(U)}.
\]
For matrix- or multilinear-form-valued functions, use the corresponding
operator norm inside the supremum. On a closed set $U$, derivatives
are taken from a smooth extension of $G$ to an open neighborhood of $U$.
We also write $\|G\|_{L^1(U)}=\int_U|G(u)|\,du$ and
$\operatorname{dist}(E,F)=\inf_{x\in E,\,y\in F}|x-y|$
for subsets $E,F\subset\mathbb R^m$.
The notation $U\ll_LV$ means $|U|\le CL^BV$ for fixed constants
$C,B$; $U\asymp_LV$ means both inequalities. These constants may
depend on the fixed smooth test and cutoff functions introduced
below and on the indicated derivative orders. Every $o(1)$ below is uniform in the allowed
parameters. We take $\varepsilon=10^{-4}$ throughout.
The signature of a real nonsingular symmetric matrix is the number
of positive eigenvalues minus the number of negative ones.

\section{Reduction of the pair correlation statistic}
\label{sec:metric}

\subsection{Smoothing and small pairs}

Let $g\in C_c^\infty(\mathbb R)$ be even and nonnegative. Write
\[
 \widehat g(\xi)=\int_{\mathbb R}g(x)\ee(-x\xi)\,dx,\qquad
 R_N(g;\alpha)=\frac1N\sum_{1\le m\ne n\le N}\sum_{\ell\in\mathbb Z}
 g\bigl(N(\alpha(f(m)-f(n))+\ell)\bigr).
\]
We abbreviate the interval statistic for $(\alpha f(n))$ as
$R_N(s;\alpha)$. Evenness of $g$, with $\ell$ replaced by $-\ell$
for the reversed pair, gives
\begin{equation}\label{eq:metric-positive-gaps}
 R_N(g;\alpha)=\frac2N
 \sum_{\substack{n,h\ge1\\n+h\le N}}\sum_{\ell\in\mathbb Z}
 g\bigl(N(\alpha G_h(n)+\ell)\bigr).
\end{equation}
We work on a fixed compact interval $I=[a,b]\subset(0,\infty)$.
A countable cover and the identity $R_N(g;-\alpha)=R_N(g;\alpha)$
will recover almost every real $\alpha$.

Remove the pairs
\[
 \mathcal B_N=\{(n,h)\in\mathbb Z_{\ge1}^2:n+h\le N,\ 
              n\le N/L^4\ \text{or}\ h\le N/L^4\}.
\]
Let $\mathcal G_N$ be the remaining pairs, and let
$R_{N,\mathcal B_N}(g;\alpha)$ denote the contribution of
$\mathcal B_N$ to \eqref{eq:metric-positive-gaps}. Then
\[
 \#\mathcal B_N\le 2N\lfloor N/L^4\rfloor\le 2N^2L^{-4}.
\]
For a fixed pair,
$G_h(n)=\int_n^{n+h}(1+\log x)\,dx\ge h\ge1$. Periodicity and
nonnegativity give
\[
 \int_0^1\sum_{\ell\in\mathbb Z}g(N(t+\ell))\,dt
 =\frac1N\int_{\mathbb R}g(x)\,dx.
\]
After the change of variable $t=\alpha G_h(n)$, cover
$[aG_h(n),bG_h(n)]$ by at most $|I|G_h(n)+2$ unit intervals.
It follows that
\[
 \begin{split}
 \int_I\sum_{\ell\in\mathbb Z}g\bigl(N(\alpha G_h(n)+\ell)\bigr)\,d\alpha
 &=\frac1{G_h(n)}\int_{aG_h(n)}^{bG_h(n)}
             \sum_{\ell\in\mathbb Z}g(N(t+\ell))\,dt\\
 &\le\frac{|I|+2/G_h(n)}N\int_{\mathbb R}g
 \le\frac{|I|+2}N\int_{\mathbb R}g.
 \end{split}
\]
Thus
\begin{equation}\label{eq:metric-small-pairs}
 \int_I R_{N,\mathcal B_N}(g;\alpha)\,d\alpha
 \le\frac{4(|I|+2)}{L^4}\int_{\mathbb R}g.
\end{equation}

We use the same subsequence for this first-moment estimate and for
the second-moment estimate below. Take
$N_k=\lfloor\exp(\sqrt{k})\rfloor$ for sufficiently large $k$. It satisfies
\begin{equation}\label{eq:metric-common-subsequence}
 \frac{N_{k+1}}{N_k}\longrightarrow1,\qquad
 \sum_k(\log N_k)^{-4}<\infty,\qquad
 \sum_kN_k^{-\eta}(\log N_k)^B<\infty
 \quad(\eta>0,\ B\ \text{fixed}).
\end{equation}
For each integer $r\ge1$, Markov's inequality yields
\[
 \sum_k\bigl|\{\alpha\in I:
        R_{N_k,\mathcal B_{N_k}}(g;\alpha)>1/r\}\bigr|
 \le r\sum_k\int_I R_{N_k,\mathcal B_{N_k}}(g;\alpha)\,d\alpha<\infty.
\]
Borel--Cantelli, followed by a countable intersection over $r$, gives
\begin{equation}\label{eq:metric-small-limit}
 R_{N_k,\mathcal B_{N_k}}(g;\alpha)\longrightarrow0
 \quad\text{for almost every }\alpha\in I.
\end{equation}
We shall restore these pairs before passing from $N_k$ to all $N$.

\subsection{Fourier expansion and truncation}

Poisson summation gives the absolutely convergent Fourier series
\[
 \sum_{\ell\in\mathbb Z}g(N(t+\ell))
   =\frac1N\sum_{j\in\mathbb Z}\widehat g(j/N)\ee(jt).
\]
Substitution into the contribution of $\mathcal G_N$ gives
\begin{equation}\label{eq:metric-fourier}
 R_N(g;\alpha)-R_{N,\mathcal B_N}(g;\alpha)
 =\frac2{N^2}\sum_{j\in\mathbb Z}\widehat g(j/N)
             \sum_{(n,h)\in\mathcal G_N}\ee(\alpha jG_h(n)).
\end{equation}
The zero-frequency term is
\begin{equation}\label{eq:metric-zero-mode}
 \begin{split}
 \frac{2\widehat g(0)}{N^2}\#\mathcal G_N
 &=\frac{2\widehat g(0)}{N^2}
       \left(\frac{N(N-1)}2-\#\mathcal B_N\right)\\
 &=\widehat g(0)-\frac{\widehat g(0)}N
       -\frac{2\widehat g(0)}{N^2}\#\mathcal B_N
 =\int_{\mathbb R}g+O_g(L^{-4}).
 \end{split}
\end{equation}
This error is independent of $\alpha$.

Both $g$ and $\widehat g$ are real and even. Combining positive
and negative frequencies therefore gives the coefficient $4/N^2$.
For each fixed $r>1$, rapid decay gives the uniform tail bound
\[
 \frac4{N^2}\#\mathcal G_N
       \sum_{j>N^{1+\varepsilon}}|\widehat g(j/N)|
 \ll_{g,r}N^{1-\varepsilon(r-1)}.
\]
Given $A>0$, choose $r$ large enough to make this $O_{g,A}(N^{-A})$.
Define the complex sum
\begin{equation}\label{eq:metric-error-definition}
 E_N(\alpha)=\frac4{N^2}
 \sum_{1\le j\le N^{1+\varepsilon}}\widehat g(j/N)
       \sum_{(n,h)\in\mathcal G_N}\ee(\alpha jG_h(n)).
\end{equation}
Equations \eqref{eq:metric-fourier}--\eqref{eq:metric-zero-mode} yield
\begin{equation}\label{eq:metric-centering}
 R_N(g;\alpha)
 =\int_{\mathbb R}g+R_{N,\mathcal B_N}(g;\alpha)
       +\operatorname{Re}E_N(\alpha)
       +O_g(L^{-4})+O_{g,A}(N^{-A}).
\end{equation}
It remains to prove, for some fixed $\eta>0$ and $B$,
\begin{equation}\label{eq:metric-target}
 \int_I|E_N(\alpha)|^2\,d\alpha\ll_{I,g}N^{-\eta}L^B.
\end{equation}
Together with \eqref{eq:metric-common-subsequence}, this will imply
$E_{N_k}(\alpha)\to0$ almost everywhere on the same subsequence as
\eqref{eq:metric-small-limit}.

\subsection{Dyadic blocks and the second moment}

Partition $n,h,j$ into disjoint half-open ranges
\[
 P\le n<2P,\qquad H\le h<2H,\qquad J\le j<2J,
 \qquad P,H,J\in\{1,2,4,\ldots\}.
\]
Only blocks with
\[
 N/(2L^4)<P,H\le N,\qquad 1\le J\le N^{1+\varepsilon}
\]
can contribute. Retain all the original restrictions by defining
\begin{equation}\label{eq:metric-block-definition}
 F_{P,H,J}(\alpha)=
 \sum_{\substack{J\le j<2J\\j\le N^{1+\varepsilon}}}\widehat g(j/N)
 \sum_{\substack{P\le n<2P,\ H\le h<2H\\(n,h)\in\mathcal G_N}}
          \ee(\alpha jG_h(n)).
\end{equation}
In particular, the inner sum still requires
$n,h>N/L^4$ and $n+h\le N$. We have the exact identity
\[
 E_N(\alpha)=\frac4{N^2}\sum_{P,H,J}F_{P,H,J}(\alpha).
\]
The number of blocks satisfies
\[
 \sum_{P,H,J}1
 \le(2+4\log_2L)^2\bigl(1+(1+\varepsilon)\log_2N\bigr)
 \ll L(\log L)^2.
\]
Cauchy's inequality now gives
\[
 |E_N(\alpha)|^2
 \le\frac{16}{N^4}\left(\sum_{P,H,J}1\right)
                  \sum_{P,H,J}|F_{P,H,J}(\alpha)|^2.
\]
After integration,
\begin{equation}\label{eq:metric-block-cauchy}
 \int_I|E_N(\alpha)|^2\,d\alpha
 \ll\frac{L(\log L)^2}{N^4}
            \sum_{P,H,J}\int_I|F_{P,H,J}(\alpha)|^2\,d\alpha.
\end{equation}
We next estimate one block uniformly in $P,H,J$. A bound
$\int_I|F_{P,H,J}|^2\ll_{I,g}N^{4-\eta}L^B$ would give
\eqref{eq:metric-target}, with an additional factor
$L^2(\log L)^4$ from the block count and Cauchy's inequality.

\section{From the second moment to four-point counting}
\label{sec:moment-proof}

Fix a compact interval $I=[a,b]\subset(0,\infty)$ and a block
$F_{P,H,J}$ from Section~\ref{sec:metric}. All restrictions in its
definition are retained below. In particular,
$N/(2L^4)<P,H\le N$, $1\le J\le N^{1+\varepsilon}$, and
$(n,h)\in\mathcal G_N$.

\subsection{Mellin separation and Parseval}

For sufficiently large $N$, the identity
\[
 G_h(n)=\int_n^{n+h}(1+\log x)\,dx
\]
and $n>N/L^4$, $n+h\le N$ give, for $n\le x\le n+h$,
\[
 \frac L2\le1+L-4\log L\le1+\log x\le1+L\le2L.
\]
Integrating and using $H\le h<2H$ and $J\le j<2J$ yields
\[
 \frac12HL\le G_h(n)\le4HL,
 \qquad \frac12JHL\le jG_h(n)\le8JHL.
\]
Choose a fixed function $\psi_M\in C_c^\infty((0,\infty))$ with
$0\le\psi_M\le1$, equal to one on $[1/2,8]$ and supported in
$(1/4,16)$. Put
\[
 K_J(\alpha,t)=\int_0^\infty
       \psi_M\!\left(\frac{x}{JHL}\right)
       \ee(\alpha x)x^{-it}\,\frac{dx}{x}.
\]
Mellin inversion (see Flajolet, Gourdon and
Dumas~\cite[Theorem~2(i)]{FGD}),
with this sign convention, reads
\[
 \psi_M\!\left(\frac{x}{JHL}\right)\ee(\alpha x)
       =\frac1{2\pi}\int_{\mathbb R}K_J(\alpha,t)x^{it}\,dt.
\]
Since $\psi_M(jG_h(n)/(JHL))=1$ and
$(jG_h(n))^{it}=j^{it}G_h(n)^{it}$, define
\[
 \begin{split}
 D_J(t)&=\sum_{\substack{J\le j<2J\\ j\le N^{1+\varepsilon}}}
       \widehat g(j/N)j^{it},\\
 Z_{P,H}(t)&=\sum_{\substack{P\le n<2P,\ H\le h<2H\\
                          (n,h)\in\mathcal G_N}}G_h(n)^{it}.
 \end{split}
\]
Exchanging the finite sums with the Mellin integral gives
\[
 F_{P,H,J}(\alpha)=\frac1{2\pi}\int_{\mathbb R}
       K_J(\alpha,t)D_J(t)Z_{P,H}(t)\,dt.
\]

The phase in $K_J(\alpha,t)$ is $2\pi\alpha x-t\log x$.
On the support of $x\mapsto\psi_M(x/(JHL))$, we have
$JHL/4<x<16JHL$.
Set
\[
 c_I=\frac{\pi a}{4},\qquad C_I=64\pi b.
\]
For $\alpha\in I$ and $t\notin[c_IJHL,C_IJHL]$,
\[
 |2\pi\alpha x-t|\gg_I JHL+|t|.
\]
The first integration by parts is
\[
 K_J(\alpha,t)
 =-\int_0^\infty\frac{d}{dx}
 \left[\frac{\psi_M(x/(JHL))}{i(2\pi\alpha x-t)}\right]
 \exp\bigl(i(2\pi\alpha x-t\log x)\bigr)\,dx.
\]
There is no boundary term because
$x\mapsto\psi_M(x/(JHL))$ is compactly supported in $(0,\infty)$.
To iterate, use $u=\log(x/(JHL))$. The transformed cutoff
$u\mapsto\psi_M(e^u)$ has fixed compact support in
$(\log(1/4),\log16)$ and uniformly bounded derivatives.
On the support of $u\mapsto\psi_M(e^u)$, the phase has first derivative of magnitude
$\gg_I JHL+|t|$, and its derivatives of every fixed higher order
are $O_I(JHL)$. Each integration by parts therefore gains a factor
$(JHL+|t|)^{-1}$, with constants independent of the block.
For every fixed integer $r\ge1$, uniformly for $\alpha\in I$
and $t\notin[c_IJHL,C_IJHL]$,
\[
 |K_J(\alpha,t)|\ll_{I,r}(JHL+|t|)^{-r}.
\]
For $r>1$, the bound $|D_J(t)Z_{P,H}(t)|\ll_g JPH$ gives an
omitted integral of size $O_{I,g,r}(JPH(JHL)^{1-r})$.
For $r\ge3$, the bounds $H>N/(2L^4)$, $P\le N$ and $J\ge1$ give
\[
 JPH(JHL)^{1-r}
 =PH^{2-r}J^{2-r}L^{1-r}
 \ll_r N^{3-r}L^{3r-7}.
\]
Thus any fixed integer $r>A+3$ makes the omitted integral
$O_{I,g,A}(N^{-A})$, uniformly in $P,H,J$.

For all $\alpha\in\mathbb R$, define
\[
 F_{P,H,J}^{\mathrm{loc}}(\alpha)
   =\frac1{2\pi}\int_{c_IJHL}^{C_IJHL}
       K_J(\alpha,t)D_J(t)Z_{P,H}(t)\,dt.
\]
Thus $F_{P,H,J}=F_{P,H,J}^{\mathrm{loc}}+O_{I,g,A}(N^{-A})$
uniformly on $I$. Substituting the kernel gives
\[
 F_{P,H,J}^{\mathrm{loc}}(\alpha)
 =\int_0^\infty\frac{\psi_M(x/(JHL))}{x}
     \left[\frac1{2\pi}\int_{c_IJHL}^{C_IJHL}
       D_J(t)Z_{P,H}(t)x^{-it}\,dt\right]\ee(\alpha x)\,dx.
\]
Fourier Parseval in $\alpha$ yields
\[
 \int_{\mathbb R}|F_{P,H,J}^{\mathrm{loc}}(\alpha)|^2\,d\alpha
 =\int_0^\infty\frac{|\psi_M(x/(JHL))|^2}{x^2}
     \left|\frac1{2\pi}\int_{c_IJHL}^{C_IJHL}
       D_J(t)Z_{P,H}(t)x^{-it}\,dt\right|^2dx.
\]
Parseval in $\log x$ gives
\[
 \int_0^\infty
     \left|\frac1{2\pi}\int_{c_IJHL}^{C_IJHL}
       D_J(t)Z_{P,H}(t)x^{-it}\,dt\right|^2\frac{dx}{x}
 =\frac1{2\pi}\int_{c_IJHL}^{C_IJHL}|D_J(t)Z_{P,H}(t)|^2\,dt.
\]
The bound $|\psi_M(x/(JHL))|^2/x\le4/(JHL)$ now implies
\[
 \int_{\mathbb R}|F_{P,H,J}^{\mathrm{loc}}(\alpha)|^2\,d\alpha
 \le\frac{2}{\pi JHL}
       \int_{c_IJHL}^{C_IJHL}|D_J(t)Z_{P,H}(t)|^2\,dt.
\]
The approximation on $I$ also gives
\[
 \int_I|F_{P,H,J}(\alpha)|^2\,d\alpha
 \le2\int_{\mathbb R}|F_{P,H,J}^{\mathrm{loc}}(\alpha)|^2\,d\alpha
       +O_{I,g,A}(N^{-A}),
\]
after increasing the order in the tail estimate if necessary.
Consequently,
\begin{equation}\label{eq:metric-local-mellin}
 \begin{split}
 \int_I|F_{P,H,J}(\alpha)|^2\,d\alpha
 \ll_I{}&\frac1{JHL}
       \sup_{c_IJHL\le t\le C_IJHL}|D_J(t)|^2
       \int_{c_IJHL}^{C_IJHL}|Z_{P,H}(t)|^2\,dt\\
 &+O_{I,g,A}(N^{-A}).
 \end{split}
\end{equation}

\subsection{The frequency sum and the gap sum}

On $[J,2J]$, the phase $t\log x/(2\pi)$ has third derivative
$t/(\pi x^3)\asymp t/J^3$. The third derivative test
(see Robert~\cite[Section~3.3]{Robert}) gives
\[
 \sup_{J\le u<2J}\left|\sum_{J\le j\le u}j^{it}\right|
       \ll J^{1/2}t^{1/6}+Jt^{-1/6},\qquad t>0.
\]
The rapid decay of $\widehat g$ and its derivative gives
\begin{equation}\label{eq:metric-coefficients}
 \sup_{J\le x\le2J}|\widehat g(x/N)|
       +\frac1N\int_J^{2J}|\widehat g'(x/N)|\,dx
       \ll_g(1+J/N)^{-2}.
\end{equation}
Abel summation, with upper endpoint at most
$\min(2J,N^{1+\varepsilon})$, therefore gives
\begin{equation}\label{eq:metric-frequency}
 \begin{split}
 \sup_{c_IJHL\le t\le C_IJHL}|D_J(t)|^2
 \ll_{I,g}(1+J/N)^{-4}
       \left[J(JHL)^{1/3}+J^2(JHL)^{-1/3}\right].
 \end{split}
\end{equation}

For the other factor, expand $|Z_{P,H}(t)|^2$ and integrate each
term. Since both gap values lie in $[HL/2,4HL]$,
\[
 \frac{|G_h(n)-G_y(m)|}{4HL}
 \le\left|\log\frac{G_h(n)}{G_y(m)}\right|
 \le\frac{2|G_h(n)-G_y(m)|}{HL}.
\]
It follows that
\[
 \left|\int_{c_IJHL}^{C_IJHL}
       \exp\!\left(it\log\frac{G_h(n)}{G_y(m)}\right)dt\right|
 \ll_I HL\min\left\{J,\frac1{|G_h(n)-G_y(m)|}\right\}.
\]
At zero gap difference, the minimum is interpreted as $J$. Hence
\[
 \int_{c_IJHL}^{C_IJHL}|Z_{P,H}(t)|^2\,dt
 \ll_I HL\sum_{\substack{P\le n,m<2P,\ H\le h,y<2H\\
                        (n,h),(m,y)\in\mathcal G_N}}
       \min\left\{J,\frac1{|G_h(n)-G_y(m)|}\right\}.
\]
Define the four-point count
\[
 \mathcal E_{P,H}(N^{-1})
 =\#\left\{\begin{array}{l}
 P\le n,m<2P,\quad H\le h,y<2H,\\
 |G_h(n)-G_y(m)|\le N^{-1}
 \end{array}\right\},
\]
where all variables are positive integers. No conditions
$(n,h),(m,y)\in\mathcal G_N$ are imposed in
$\mathcal E_{P,H}(N^{-1})$, so it also bounds the count with
these additional restrictions.
For $k\in\mathbb Z$, put
\[
 C_k=\#\left\{\begin{array}{l}
 (n,h)\in\mathcal G_N,\quad P\le n<2P,\quad H\le h<2H,\\
 k/N\le G_h(n)<(k+1)/N
 \end{array}\right\}.
\]
These numbers vanish unless
$\lfloor NHL/2\rfloor\le k\le\lfloor4NHL\rfloor$.
Two gap values in the same interval differ by less than $N^{-1}$.
Thus Cauchy's inequality gives, for every integer $r$,
\begin{equation}\label{eq:metric-energy-width}
 \sum_k C_kC_{k+r}
 \le\left(\sum_k C_k^2\right)^{1/2}
       \left(\sum_k C_{k+r}^2\right)^{1/2}
 =\sum_k C_k^2\le\mathcal E_{P,H}(N^{-1}).
\end{equation}
For interval indices differing by $|r|\ge2$, the gap difference
is at least $(|r|-1)/N$. The cases $r=0,\pm1$ contribute at most
$J$ per pair. Since $|r|\le4NHL$, we obtain
\[
 \begin{split}
 &\sum_{\substack{P\le n,m<2P,\ H\le h,y<2H\\
                  (n,h),(m,y)\in\mathcal G_N}}
       \min\left\{J,\frac1{|G_h(n)-G_y(m)|}\right\}\\
 &\quad\le J\sum_{r=-1}^1\sum_k C_kC_{k+r}
       +\sum_{2\le|r|\le4NHL}\min\left\{J,\frac{N}{|r|-1}\right\}
          \sum_k C_kC_{k+r}\\
 &\quad\ll\left(J+N\sum_{1\le r\le4NHL}\frac1r\right)
          \mathcal E_{P,H}(N^{-1})
 \ll(J+NL)\mathcal E_{P,H}(N^{-1}).
 \end{split}
\]
Here $H\le N$ gives $\log(2NHL)\ll L$. Therefore
\begin{equation}\label{eq:metric-gap-moment}
 \int_{c_IJHL}^{C_IJHL}|Z_{P,H}(t)|^2\,dt
       \ll_I HL(J+NL)\mathcal E_{P,H}(N^{-1}).
\end{equation}

\subsection{The required four-point estimate}

Combining \eqref{eq:metric-local-mellin},
\eqref{eq:metric-frequency}, and \eqref{eq:metric-gap-moment} gives
\[
 \begin{split}
 \int_I|F_{P,H,J}(\alpha)|^2\,d\alpha
 \ll_{I,g}{}&(1+J/N)^{-4}(1+NL/J)
       \left[J^{4/3}(HL)^{1/3}+J^{5/3}(HL)^{-1/3}\right]
       \mathcal E_{P,H}(N^{-1})\\
 &+O_{I,g,A}(N^{-A}).
 \end{split}
\]
For $r=4/3,5/3$,
\[
 (1+J/N)^{-4}(1+N/J)J^r
 =N^r\frac{(J/N)^{r-1}}{(1+J/N)^3}\le N^r.
\]
Using $1+NL/J\le L(1+N/J)$ and
$N/(2L^4)<H\le N$, we obtain
\begin{equation}\label{eq:metric-dyadic}
 \begin{split}
 \int_I|F_{P,H,J}(\alpha)|^2\,d\alpha
 &\ll_{I,g}L\left[N^{4/3}(HL)^{1/3}
                  +N^{5/3}(HL)^{-1/3}\right]
                  \mathcal E_{P,H}(N^{-1})+O_{I,g,A}(N^{-A})\\
 &\ll_{I,g}N^{5/3}L^{4/3}\mathcal E_{P,H}(N^{-1})
                  +O_{I,g,A}(N^{-A}).
 \end{split}
\end{equation}
Indeed, the two coefficients in the first line are at most
$N^{5/3}L^{4/3}$ and $2^{1/3}N^{4/3}L^2$, respectively.
The second is smaller than the first for sufficiently large $N$.
All constants are uniform in $P,H,J$.

There are $O_\varepsilon(L(\log L)^2)$ blocks.
Substituting \eqref{eq:metric-dyadic} into
\eqref{eq:metric-block-cauchy} therefore gives
\begin{equation}\label{eq:metric-moment}
 \int_I|E_N(\alpha)|^2\,d\alpha
 \ll_{I,g,\varepsilon}N^{-7/3}L^{10/3}(\log L)^4
       \max_{P,H}\mathcal E_{P,H}(N^{-1})+O_{I,g,A}(N^{-A}).
\end{equation}
Thus a fixed power saving over $N^{7/3}$ in this count is sufficient.
The required estimate is the following.

\begin{proposition}\label{prop:energy}
There is an absolute $\sigma>0$ such that, uniformly for
$N/(2L^4)\le P,H\le N$,
\begin{equation}\label{eq:energy}
 \mathcal E_{P,H}(N^{-1})\ll N^{7/3-\sigma}.
\end{equation}
The estimate also holds when the conditions $P\le n,m<2P$ and
$H\le h,y<2H$ in $\mathcal E_{P,H}(N^{-1})$ are replaced by
$P/C_1\le n,m\le C_1P$ and $H/C_1\le h,y\le C_1H$,
for any fixed $C_1\ge2$, with a constant depending on $C_1$.
\end{proposition}

Four-point spacing estimates also occur in the metric theory of
power sequences; see Robert and Sargos~\cite{RobertSargos} and
Aistleitner, El-Baz and Munsch~\cite{AEM}.
The tolerance $N^{-1}$ is essential here. For $P\asymp H\asymp N$,
there are $\asymp N^2$ gaps in an interval of length $O(NL)$.
Cauchy's inequality on unit intervals therefore gives a fixed-scale
energy of size at least $N^3/L$.

Proposition~\ref{prop:energy} is proved in
Sections~\ref{sec:first}--\ref{sec:remaining-shifts}.
Together with \eqref{eq:metric-moment}, it yields
\begin{equation}\label{eq:metric-variance}
 \int_I|E_N(\alpha)|^2\,d\alpha
       \ll_{I,g,\varepsilon}N^{-\sigma}L^C
\end{equation}
for some fixed $C>0$, after decreasing $\sigma$ if necessary.

\subsection{Proof of the main theorem}

\begin{proof}[Proof of Theorem~\ref{thm:main}]
Assume Proposition~\ref{prop:energy}, proved in
Sections~\ref{sec:first}--\ref{sec:remaining-shifts}.
We use the standard passage from mean-square estimates to
almost everywhere convergence; see Rudnick and
Technau~\cite[Section~2.5]{RTSmallPowers}.
Fix $I=[a,b]\subset(0,\infty)$ and an even nonnegative
$g\in C_c^\infty(\mathbb R)$.
By \eqref{eq:metric-variance} and the common subsequence
\eqref{eq:metric-common-subsequence},
\[
 \sum_k\int_I|E_{N_k}(\alpha)|^2\,d\alpha
 \ll_{I,g}\sum_kN_k^{-\sigma}(\log N_k)^C<\infty.
\]
For every integer $r\ge1$, Markov's inequality gives
\[
 \sum_k\bigl|\{\alpha\in I:|E_{N_k}(\alpha)|>1/r\}\bigr|
 \le r^2\sum_k\int_I|E_{N_k}(\alpha)|^2\,d\alpha<\infty.
\]
Borel--Cantelli and the countable intersection over $r$ imply
$E_{N_k}(\alpha)\to0$ almost everywhere on $I$.
Combining this with \eqref{eq:metric-small-limit} and
\eqref{eq:metric-centering} restores the full smoothed statistic:
\begin{equation}\label{eq:metric-test-limit}
 R_{N_k}(g;\alpha)\longrightarrow\int_{\mathbb R}g
 \quad\text{for almost every }\alpha\in I.
\end{equation}
This conclusion holds simultaneously for any fixed countable
family of test functions, after intersecting the full-measure sets.

For each positive rational $s$ and integer $r\ge1$, choose even
nonnegative $g_{s,r}^{\pm}\in C_c^\infty(\mathbb R)$ with
\[
 g_{s,r}^{-}\le\mathbf1_{[-s,s]}\le g_{s,r}^{+},
 \qquad
 2s-r^{-1}\le\int g_{s,r}^{-}
       \le\int g_{s,r}^{+}\le2s+r^{-1}.
\]
Apply \eqref{eq:metric-test-limit} to this countable family.
For $N_k>2s$, periodizing the indicator gives exactly the
condition defining $R_{N_k}(s;\alpha)$. Hence
\[
 R_{N_k}(g_{s,r}^{-};\alpha)
 \le R_{N_k}(s;\alpha)
 \le R_{N_k}(g_{s,r}^{+};\alpha).
\]
First let $k\to\infty$, then $r\to\infty$. On one full-measure
subset of $I$, this proves $R_{N_k}(s;\alpha)\to2s$ for all positive
rational $s$. Monotonicity in $s$ gives the same conclusion for
every real $s>0$.

To pass to all $N$, use the full interval count. If
$N_k\le N\le N_{k+1}$, inclusion of the index sets and of the
distance windows gives
\begin{equation}\label{eq:metric-full-sequence}
 \frac{N_k}{N_{k+1}}
 R_{N_k}\!\left(s\frac{N_k}{N_{k+1}};\alpha\right)
 \le R_N(s;\alpha)
 \le\frac{N_{k+1}}{N_k}
 R_{N_{k+1}}\!\left(s\frac{N_{k+1}}{N_k};\alpha\right).
\end{equation}
The left count uses the first $N_k$ indices with window
$s/N_{k+1}$, and the right count uses the first $N_{k+1}$ indices
with window $s/N_k$. Fix positive rational numbers $s_-<s<s_+$.
Since $N_{k+1}/N_k\to1$, for all sufficiently large $k$,
\[
 s_-<s\frac{N_k}{N_{k+1}},\qquad
 s\frac{N_{k+1}}{N_k}<s_+.
\]
Monotonicity in the window parameter and \eqref{eq:metric-full-sequence}
therefore give
\[
 2s_-\le\liminf_{N\to\infty}R_N(s;\alpha)
 \le\limsup_{N\to\infty}R_N(s;\alpha)\le2s_+.
\]
Letting $s_-\uparrow s$ and $s_+\downarrow s$ proves the limit for all $N$.

Finally, use the countable cover
$(0,\infty)=\bigcup_{r\ge2}[1/r,r]$ and
$R_N(s;-\alpha)=R_N(s;\alpha)$.
This gives the conclusion for almost every real $\alpha$ and
simultaneously for every $s>0$.
Remove also the countable set
\[
 \bigcup_{1\le m<n}\left\{\frac{\ell}{f(n)-f(m)}:\ell\in\mathbb Z\right\}.
\]
For every remaining $\alpha$, $R_N(0;\alpha)=0$ for all $N$.
\end{proof}

\section{The first transformation}\label{sec:first}

We now prove Proposition~\ref{prop:energy}.
Fix $N/(2L^4)\le P,H\le N$ and set $d=h-y$.
The terms with $d=0$ contribute $O(PH)$, by the observation following
\eqref{app:near-integer}. Interchanging $(n,h)$ and $(m,y)$
reduces the other terms to $d>0$.

For $0<d<h$, let $\kappa(d,n,h)$ be the unique positive solution of
\begin{equation}\label{first:root}
 G_{h-d}\bigl(\kappa(d,n,h)\bigr)=G_h(n).
\end{equation}
By \eqref{app:near-integer},
\begin{equation}\label{first:reduction}
 |m-\kappa(d,n,h)|\ll\frac{L^4}{N},\qquad
 d\ll\frac{P\log(1+H/P)}{L}\ll\frac HL.
\end{equation}
Thus each triple $(d,n,h)$ contributes at most one integer $m$.

Split $d>0$ into dyadic ranges $D\le d<2D$. Choose fixed nonnegative
functions $\theta_d,\theta_n,\theta_h\in C_c^\infty((0,\infty))$,
each at least one on $[1,2]$, and put
\[
 \omega(d,n,h)
 =\theta_d(d/D)\theta_n(n/P)\theta_h(h/H).
\]
Choose a fixed $C_\Omega>1$ such that $\supp\omega$ is compactly
contained in the $(d,n,h)$ box
\[
 \Omega_{D,P,H}
 =\left(\frac D{C_\Omega},C_\Omega D\right)
  \times\left(\frac P{C_\Omega},C_\Omega P\right)
  \times\left(\frac H{C_\Omega},C_\Omega H\right).
\]
By \S\ref{app:root-range}, $\kappa$ is smooth on
$\Omega_{D,P,H}$ and $\kappa\asymp P$ there.
The amplitude $\omega$ satisfies
\begin{equation}\label{first:omega}
 D^iP^jH^k
 |\partial_d^i\partial_n^j\partial_h^k\omega|\ll_{i,j,k}1.
\end{equation}

\subsection{The Fej\'er majorant}

Set $R=\lfloor N^{2/3+\varepsilon}\rfloor$.
The nonnegative Fej\'er kernel
(see Stein and Shakarchi~\cite[Chapter~2, Lemma~5.1]{SteinShakarchi})
\[
 K_R(t)=\frac1R\left|\sum_{j=0}^{R-1}\ee(jt)\right|^2
       =\sum_{|q|<R}\left(1-\frac{|q|}{R}\right)\ee(qt)
\]
satisfies $K_R(t)\gg R$ for $\|t\|\le 1/(4R)$.
Let $C_0$ be a fixed constant in the first bound of
\eqref{first:reduction}. Since $C_0RL^4/N=o(1)$, for sufficiently
large $N$ we have the pointwise bound
\[
 \mathbf 1_{\{\|t\|\le C_0L^4/N\}}\ll\frac{K_R(t)}R
 \qquad(t\in\RR).
\]
For each dyadic $D$, define
\[
 S_q=\sum_{d,n,h\in\ZZ}\omega(d,n,h)\ee\bigl(q\kappa(d,n,h)\bigr).
\]
Here and in the Poisson integrals below, extend the full weighted
summand by zero outside $\Omega_{D,P,H}$.
Each triple $(d,n,h)$ contributes at most one integer $m$, and such
an integer can occur only if $\|\kappa(d,n,h)\|\le C_0L^4/N$.
Since $\omega(d,n,h)\ge1$ for $D\le d<2D$, $P\le n<2P$,
and $H\le h<2H$, the contribution to $\mathcal E_{P,H}(N^{-1})$
from $D\le d<2D$ is at most a constant times
\[
 \begin{aligned}
 \frac1R\sum_{d,n,h\in\ZZ}\omega(d,n,h)K_R\bigl(\kappa(d,n,h)\bigr)
 &=\frac1R\sum_{|q|<R}\left(1-\frac{|q|}{R}\right)S_q\\
 &\le\frac{S_0}R+\frac2R\sum_{1\le q<R}|S_q|.
 \end{aligned}
\]
Here $S_{-q}=\overline{S_q}$ because $\omega$ is real, and
$S_0=\sum\omega\ll DPH$ by the support and boundedness of $\omega$.
Thus the block contributes at most
\begin{equation}\label{first:Fejer}
 C\left(\frac{DPH}{R}+\frac1R\sum_{1\le q<R}|S_q|\right).
\end{equation}
The dyadic sum of the zero-frequency terms is
\begin{equation}\label{first:zero}
 \sum_D\frac{DPH}{R}\ll\frac{N^3}{LR}
 \ll N^{7/3-\varepsilon}.
\end{equation}
It remains to estimate $S_q$ uniformly in
\begin{equation}\label{main:range}
 \frac{N}{2L^4}\le P,H\le N,\qquad
 1\le D\le\frac{CP\log(1+H/P)}{L}\le\frac{CH}{L},\qquad
 1\le q<R.
\end{equation}

\subsection{Poisson summation and stationary phase}

We apply Poisson summation followed by stationary phase,
namely the multidimensional van der Corput $B$-process;
see Guo~\cite[Section~2]{Guo}.

\begin{proposition}\label{first:transform}
Uniformly for the parameters $P,H,D,q$ satisfying \eqref{main:range},
\begin{equation}\label{main:firstB}
 \begin{aligned}
 S_q&=\ee(1/8)\frac{N^2}{q^{3/2}\sqrt D}\,\Sigma_q+E_q,\\
 \Sigma_q&=\sum_{a,b,s\in\ZZ}W(a,b,s)\ee\bigl(\bar\Psi(a,b,s)\bigr),\\
 |E_q|&\ll_L N^2(qD)^{-3/2}+(qD)^{1/2}.
 \end{aligned}
\end{equation}
For $(a,b,s)\in q\nabla\kappa(\Omega_{D,P,H})$, let $(d,n,h)$
be the unique solution in $\Omega_{D,P,H}$ of
$q\nabla\kappa(d,n,h)=(a,b,s)$. Define
\begin{equation}\label{main:phaseweight}
 \bar\Psi=q\kappa-ad-bn-sh,\qquad
 W=\frac{\sqrt D}{N^2}
       \frac{\omega(d,n,h)}{\sqrt{|\det\nabla^2\kappa(d,n,h)|}}.
\end{equation}
Extend $W$ smoothly by zero outside
$q\nabla\kappa(\Omega_{D,P,H})$, and interpret the full summand
as zero wherever $W$ vanishes.
\end{proposition}

\begin{proof}
Apply Proposition~\ref{prop:multidimensional-B} with phase $q\kappa$,
amplitude $\omega$, and
\[
 m=3,\qquad c=0,\qquad \mathsf X=\diag(D,P,H),\qquad
 T=qD\ge1,\qquad \xi=(0,q,0).
\]
The normalized domain is the fixed box
$(C_\Omega^{-1},C_\Omega)^3$.
The support of $\omega$ has a fixed positive margin inside
$\Omega_{D,P,H}$ in the $(D,P,H)$ scales.
The derivative hypotheses follow from \S\ref{app:coarse} and
\eqref{first:omega}. The normalized Hessian is $K_\kappa$ from
\eqref{app:K-definition}; this matrix and its inverse are $O_L(1)$.
The stationary map is injective by \S\ref{app:first-unique}, and
its Hessian has signature $1$ by \S\ref{app:root-range}.

For the centered gradient bound, \S\ref{app:coarse} gives a fixed
$A=CL^B\ge1$ such that throughout $\Omega_{D,P,H}$,
\[
 |\partial_d\kappa|\le A,\qquad
 |\partial_n\kappa-1|\le AD/P,\qquad
 |\partial_h\kappa|\le AD/H.
\]
Thus the counting box may be taken to be
\[
 \mathcal B_q=\{(a,b,s)\in\mathbb R^3:
 |a|\le2Aq,\quad |b-q|\le2AqD/P,\quad |s|\le2AqD/H\}.
\]
It contains $q\nabla\kappa(\Omega_{D,P,H})$ and has
$O_L(q(1+qD/P)(1+qD/H))$ integer points.
The center $(0,q,0)$ corresponds to writing the phase as
$q(\kappa-n)-ad-(b-q)n-sh$.

The leading term in \eqref{eq:general-B-formula}, using
\eqref{main:phaseweight}, is exactly
$\ee(1/8)N^2\Sigma_q/(q^{3/2}\sqrt D)$.
By \eqref{eq:general-B-error} and $q\ge1$,
\begin{align*}
 |E_q|
 &\ll_L DPH(qD)^{-5/2}q(1+qD/P)(1+qD/H)\\
 &=PH(qD)^{-3/2}+(P+H)(qD)^{-1/2}+(qD)^{1/2}\\
 &\ll N^2(qD)^{-3/2}+(qD)^{1/2}.
\end{align*}
The last line uses $P,H\le N$ and
$2N(qD)^{-1/2}\le N^2(qD)^{-3/2}+(qD)^{1/2}$.
\end{proof}

Set
\begin{equation}\label{main:scales}
 \beta=\frac{qDL}{N},\qquad S=\operatorname{diag}(q,\beta,\beta).
\end{equation}
For $x=(a,b,s)$ and $x'=(a',b',s')$, the normalized distance used below is
\[
 |S^{-1}(x-x')|=
 \left(\frac{(a-a')^2}{q^2}
 +\frac{(b-b')^2+(s-s')^2}{\beta^2}\right)^{1/2}.
\]
Equations~\eqref{app:dual-coarse} and \eqref{app:amplitude-coarse}
give, for every fixed order,
\begin{equation}\label{main:inputbounds}
 \begin{aligned}
 q^i\beta^{j+k}
 |\partial_a^i\partial_b^j\partial_s^kW|&\ll_L1,\\
 q^i\beta^{j+k}
 |\partial_a^i\partial_b^j\partial_s^k\bar\Psi|&\ll_LqD
 \qquad (i+j+k\ge2).
 \end{aligned}
\end{equation}
The phase is defined, and these bounds hold uniformly, whenever
\begin{equation}\label{first:legal}
 \inf_{x_0\in\supp W}|S^{-1}(x-x_0)|<L^{-B_*},
 \qquad x=(a,b,s),
\end{equation}
for a fixed $B_*>0$. The corresponding stationary point
$(d,n,h)$ lies in $\Omega_{D,P,H}$.

The support $\supp W$ lies in $\mathcal B_q$, defined in the proof
of Proposition~\ref{first:transform}.
Since $P,H\ge N/(2L^4)$, the box $\mathcal B_q$ has
$O_L(q(1+\beta)^2)$ integer points.
We use $\mathcal B_q$ only for counting; the phase $\ps$ is evaluated
on the neighborhood \eqref{first:legal} of $\supp W$.

\subsection{The range \texorpdfstring{$qD\le N^{3/2}$}{qD at most N to the power 3/2}}

Taking absolute values in the first transformed sum gives
\begin{align*}
 \frac{N^2}{q^{3/2}\sqrt D}|\Sigma_q|
 &\ll_L\frac{N^2}{q^{3/2}\sqrt D}\,q(1+\beta)^2\\
 &=\frac{N^2}{\sqrt{qD}}+2NL\sqrt{qD}+L^2(qD)^{3/2}.
\end{align*}
The middle term is at most the sum of the first and third terms.
The error in \eqref{main:firstB} is also absorbed, since $qD\ge1$.
Hence
\begin{equation}\label{main:small}
 |S_q|\ll_L(qD)^{3/2}+\frac{N^2}{\sqrt{qD}}
       \ll_L N^{9/4}\qquad(1\le qD\le N^{3/2}).
\end{equation}
The dyadic sum of these contributions to \eqref{first:Fejer} is
$O_L(N^{9/4})$.

\section{Differencing and derivative estimates}\label{sec:differencing}
Assume $qD>N^{3/2}$ and put
\begin{equation}\label{main:R}
 \mathcal R=\{u\in\ZZ^3: |u_1|\le N^{57/50}/(DL),\quad
 |u_2|,|u_3|\le N^{1/10}\}.
\end{equation}
Then $\#\mathcal R\asymp N^{67/50}/(DL)$. Every difference
$\ell\in\mathcal R-\mathcal R$ satisfies
\begin{equation}\label{main:shifts}
 |\ell_1|\le 2N^{57/50}/(DL),\quad
 |\ell_2|,|\ell_3|\le2N^{1/10},\quad
 |S^{-1}\ell|_\infty\ll N^{-9/25}/L.
\end{equation}
Indeed, $|\ell_1|/q\ll N^{57/50}/(qDL)$ and
$(|\ell_2|+|\ell_3|)/\beta\ll N^{11/10}/(qDL)$.
For the localization in \S\ref{sec:partition}, we establish the
following derivative estimates for all $x$ satisfying
\[
 \inf_{x'\in\supp W}|S^{-1}(x-x')|<\tfrac12L^{-B_*}.
\]
By \eqref{main:shifts}, every segment $x+t\ell$, $0\le t\le1$,
with such a starting point $x$ lies in the neighborhood
\eqref{first:legal} of $\supp W$.
Thus $\ps$ is defined along $x+t\ell$, and the corresponding
stationary point $(d,n,h)$ remains in $\Omega_{D,P,H}$.

Write
\begin{equation}\label{main:correlations}
 W_\ell(x)=W(x+\ell)\overline{W(x)},\qquad
 \ph(x)=\ps(x+\ell)-\ps(x),\qquad
 C_\ell=\sum_{x\in\ZZ^3}W_\ell(x)\ee(\ph(x)).
\end{equation}
Each summand defining $C_\ell$ is extended by zero outside
$\supp W_\ell$.
Averaging the translated sums gives the exact identity
\[
 (\#\mathcal R)\Sigma_q
 =\sum_{x\in\mathbb Z^3}\sum_{u\in\mathcal R}
 W(x+u)\ee(\ps(x+u)).
\]
The outer sum has $O_L(q\beta^2)$ possibly nonzero terms.
Cauchy's inequality and expansion of the square therefore give
\[
 |\Sigma_q|^2\ll_L
 \frac{q\beta^2}{(\#\mathcal R)^2}
 \sum_{u,v\in\mathcal R}|C_{u-v}|.
\]
Each difference has at most $\#\mathcal R$ representations, and
$C_0=\sum_x|W(x)|^2\ll_Lq\beta^2$. This proves the
usual differencing inequality
(see Graham and Kolesnik~\cite{GK}):
\begin{equation}\label{main:Aineq}
 |\Sigma_q|^2\ll_L\frac{q\beta^2}{\#\mathcal R}
 \left(q\beta^2+\sum_{\substack{\ell\in\mathcal R-\mathcal R\\\ell\ne0}}
 |C_\ell|\right).
\end{equation}
Combining this with \eqref{main:firstB} gives
\begin{equation}\label{main:Aentry}
 |S_q|\ll_L\frac{(qD)^{3/2}}{\sqrt{\#\mathcal R}}
 +N\sqrt D\left(\frac1{\#\mathcal R}\sum_{\ell\ne0}|C_\ell|\right)^{1/2}
 +|E_q|.
\end{equation}

To estimate the correlations $C_\ell$, we use the second $B$-process
in three or two variables according to the available Hessian bounds.
For the $a$-dominant shifts defined in \eqref{main:dominant},
\S\ref{sec:convex-finish} applies the three-dimensional $B$-process
followed by a third-derivative estimate.
For the remaining shifts, \S\ref{sec:remaining-shifts} combines a
three-dimensional estimate where the full Hessian determinant is
sufficiently large with a two-dimensional estimate in $(a,b)$ or
$(a,s)$ near degeneracy, and then averages over the shifts.

For a nonzero difference put
\begin{equation}\label{main:F}
 F_\ell=\left(D^2\ell_1^2+\frac{N^2}{L^2}(\ell_2^2+\ell_3^2)\right)^{1/2}.
\end{equation}
For every fixed $i+j+k\ge2$, the segment formula and
\eqref{main:inputbounds} give
\begin{equation}\label{main:shiftbounds}
 q^i\beta^{j+k}|\partial_a^i\partial_b^j\partial_s^k\ph|
 \ll_L F_\ell.
\end{equation}
In particular,
\begin{equation}\label{main:hessian}
 \|S\nabla^2\ph S\|\ll_LF_\ell.
\end{equation}

\subsection{Three fixed directions}\label{sec:lower}
Fix an arbitrary point $x_0=(a_0,b_0,s_0)$ satisfying
\[
 \inf_{x\in\supp W}|S^{-1}(x_0-x)|<\tfrac12L^{-B_*}.
\]
Subscript $0$ denotes evaluation at $x_0$. We use the quantities
$v=n+h$, $w=\kappa+h-d$, $z=\log(w/\kappa)$, and $\Delta$ from
Appendix~\ref{app:root}. The identities proved there give
\begin{equation}\label{main:dualidentities}
 \ps=-qd/z,\qquad \nabla\ps=-(d,n,h)^T,\qquad
 \partial_a\ps=(z/q)\ps,\qquad
 \nabla^2\ps=-q^{-1}(\nabla^2\kappa)^{-1}.
\end{equation}
Define the fixed vectors
\begin{equation}\label{main:frame}
 \begin{aligned}
 V_{1,0}&=\frac q{z_0}(1,0,0)^T,\\
 V_{2,0}&=q\sqrt{\frac{Lh_0}{v_0z_0^3}}(-1,d_0/n_0,0)^T,\\
 V_{3,0}&=q\sqrt{\frac{Lh_0}{v_0z_0^3}}(-1,0,d_0/h_0)^T,
 \qquad V_0=(V_{1,0},V_{2,0},V_{3,0}),
 \end{aligned}
\end{equation}
and let $\partial_i=V_{i,0}\cdot\nabla_{a,b,s}$, with coefficients fixed
throughout the neighborhood \eqref{first:legal} of $\supp W$. Then
\begin{equation}\label{main:directionpurpose}
 (\partial_1\ps)(x_0)=\ps(x_0),\qquad
 (\partial_2\ps)(x_0)=(\partial_3\ps)(x_0)=0.
\end{equation}
For $x$ satisfying \eqref{first:legal}, define
\begin{equation}\label{main:Adef}
 \mathsf A(x)=(A_{ij}(x))_{1\le i,j\le3},\qquad
 A_{ij}(x)=\frac{\partial_i\partial_j\ps(x)}{\ps(x)},
 \qquad A_{ij,0}=A_{ij}(x_0).
\end{equation}
Differentiate the formulas in Appendix~\ref{mat:appendix}, then
evaluate at $x_0$. This gives
\begin{equation}\label{main:Abounds}
 \begin{gathered}
 1\le A_{11,0}\le C,\quad A_{22,0}\le-\tfrac12,\quad
 |A_{ij,0}|\le C,\quad A_{12,0},A_{13,0}=o(1),\\
 A_{22,0}A_{33,0}-A_{23,0}^2\le-c,\\
 (\partial_1A_{ij})(x_0)=o(1)\quad(1\le i,j\le3),\qquad
 (\partial_kA_{ij})(x_0)=o(1)\quad(i,j,k\in\{2,3\}).
 \end{gathered}
\end{equation}
All constants and small terms are uniform in $x_0$.

Set
\begin{equation}\label{main:Tdef}
 T_k=\left(\frac{(\partial_i\partial_j\partial_k\ps)(x_0)}{\ps(x_0)}
 \right)_{1\le i,j\le3}.
\end{equation}
By the product rule, its $(i,j)$ entry is
$\delta_{k1}A_{ij,0}+(\partial_kA_{ij})(x_0)$,
where $\delta_{k1}$ equals $1$ if $k=1$ and $0$ otherwise.
The fixed directional derivatives commute. For example,
\[
 (T_2)_{12}=(T_1)_{22}=A_{22,0}+o(1),\qquad
 (T_2)_{11}=(T_1)_{12}=o(1).
\]
Thus entries not covered directly by the derivative bounds in
\eqref{main:Abounds} follow by symmetry of the third derivative.
In this way, \eqref{main:Abounds} yields
\begin{equation}\label{main:Tmatrices}
 \begin{aligned}
 T_1&=\begin{pmatrix}A_{11,0}&0&0\\0&A_{22,0}&A_{23,0}\\0&A_{23,0}&A_{33,0}\end{pmatrix}+o(1),\\
 T_2&=\begin{pmatrix}0&A_{22,0}&A_{23,0}\\A_{22,0}&0&0\\A_{23,0}&0&0\end{pmatrix}+o(1),\\
 T_3&=\begin{pmatrix}0&A_{23,0}&A_{33,0}\\A_{23,0}&0&0\\A_{33,0}&0&0\end{pmatrix}+o(1).
 \end{aligned}
\end{equation}

The following lemma gives, in the frame $V_0$, the principal minor
and full determinant bounds needed for the two- and three-dimensional
$B$-processes.

\begin{lemma}\label{main:combination}
For $I\subseteq\{1,2,3\}$, a subscript $I$ denotes the principal
submatrix formed from the rows and columns indexed by $I$.
There are constants $c,\delta>0$ such that, for every $r\in\RR^3$,
\begin{equation}\label{main:combination2}
 \max_{I=\{1,2\},\{1,3\}}
 \left|\det\!\left(\left(\sum_{k=1}^3r_kT_k\right)_I\right)\right|
 \ge c|r|^2.
\end{equation}
If $|r_2|+|r_3|\le\delta|r_1|$, then
\begin{equation}\label{main:combination3}
 \left|\det\left(\sum_{k=1}^3r_kT_k\right)\right|\ge c|r_1|^3.
\end{equation}
Both estimates remain valid, with a possibly smaller constant $c>0$,
after adding to $\sum_{k=1}^3r_kT_k$ any $3\times3$ matrix whose
entries are $o(|r|)$ uniformly in $x_0$ and $r\ne0$ as $N\to\infty$.
\end{lemma}
\begin{proof}
For the displayed parts of \eqref{main:Tmatrices}, the two minors are
\[
 A_{11,0}A_{22,0}r_1^2-(A_{22,0}r_2+A_{23,0}r_3)^2,
 \qquad
 A_{11,0}A_{33,0}r_1^2-(A_{23,0}r_2+A_{33,0}r_3)^2.
\]
The first has absolute value $\gg r_1^2+(A_{22,0}r_2+A_{23,0}r_3)^2$,
since $A_{22,0}\le-1/2$. The second, together with $r_1^2$, controls
$(A_{23,0}r_2+A_{33,0}r_3)^2$.
The lower right $2\times2$ matrix has a bounded inverse, so these
quantities control $|r|^2$. The determinant is
\[
 (A_{22,0}A_{33,0}-A_{23,0}^2)r_1
 [A_{11,0}r_1^2-A_{22,0}r_2^2-2A_{23,0}r_2r_3-A_{33,0}r_3^2].
\]
When $|r_2|+|r_3|\le\delta|r_1|$, the bracket is at least
$r_1^2/2$ for small fixed $\delta$. A perturbation of size $o(|r|)$ changes the minors
by $o(|r|^2)$ and the determinant by $o(|r|^3)$.
\end{proof}

\subsection{The shifted Hessian}
Write $\ell=\sum_k\tau_kV_{k,0}$, where
\begin{equation}\label{main:tau}
 \begin{aligned}
 \tau_1&=\frac{z_0}{q}\left(\ell_1+\frac{n_0\ell_2+h_0\ell_3}{d_0}\right),\\
 \tau_2&=\frac{n_0\ell_2}{qd_0}\sqrt{\frac{v_0z_0^3}{Lh_0}},\qquad
 \tau_3=\frac{h_0\ell_3}{qd_0}\sqrt{\frac{v_0z_0^3}{Lh_0}}.
 \end{aligned}
\end{equation}
The root estimates and \eqref{main:frame} give
\[
 \|S^{-1}V_0\|+\|V_0^{-1}S\|\ll_L1.
\]
Since $\tau=V_0^{-1}\ell$, it follows that
\[
 |\tau|\asymp_L|S^{-1}\ell|
 =\frac{F_\ell}{qD}\ll_LN^{-9/25}.
\]
Each $\partial_i$ is a linear combination of
$q\partial_a$, $\beta\partial_b$ and $\beta\partial_s$, with fixed
$O_L(1)$ coefficients. Since $|\ps(x_0)|\asymp_LqD$,
\eqref{main:inputbounds} along the shift segment gives
\[
 \max_{\substack{1\le i,j,k,m\le3\\0\le t\le1}}
 \left|\frac{(\partial_i\partial_j\partial_k\partial_m\ps)(x_0+t\ell)}
 {\ps(x_0)}\right|\ll_L1.
\]
Entrywise Taylor expansion along the segment gives
\begin{equation}\label{main:taylor}
 \begin{aligned}
 \frac{\partial_i\partial_j\ph(x_0)}{\ps(x_0)}
 &=\sum_{k=1}^3\tau_k(T_k)_{ij}+\frac{R^V_{ij}}{\ps(x_0)},\\
 R^V_{ij}&=\int_0^1(1-t)\sum_{k,m=1}^3\tau_k\tau_m
 (\partial_i\partial_j\partial_k\partial_m\ps)(x_0+t\ell)\,\dd t.
 \end{aligned}
\end{equation}
The remainder therefore satisfies
\[
 \max_{i,j}|R^V_{ij}/\ps(x_0)|\ll_L|\tau|^2=o(|\tau|).
\]
Set $\mathsf B_\ell=V_0^T\nabla^2\ph(x_0)V_0/\ps(x_0)$.
By \eqref{main:taylor},
\[
 \mathsf B_\ell=\sum_{k=1}^3\tau_kT_k
 +\left(\frac{R^V_{ij}}{\ps(x_0)}\right)_{1\le i,j\le3}.
\]
The preceding remainder bound allows us to apply
Lemma~\ref{main:combination} with $r=\tau$.
The left-hand sides below are the Hessian determinants of the
restrictions of $\ph$ to the $(a,b)$ and $(a,s)$ planes, with $s$
and $b$ fixed, respectively.
Since $(V_{1,0},V_{2,0})$ and $(V_{1,0},V_{3,0})$ are bases of
these planes, respectively, the changes of basis give
\begin{equation}\label{main:planefactors}
 \begin{aligned}
 [\partial_a^2\ph\,\partial_b^2\ph-(\partial_a\partial_b\ph)^2](x_0)
 &=\frac{v_0n_0^2z_0^3}{q^2Lh_0}
   \det\!\left((\mathsf B_\ell)_{\{1,2\}}\right),\\
 [\partial_a^2\ph\,\partial_s^2\ph-(\partial_a\partial_s\ph)^2](x_0)
 &=\frac{v_0h_0z_0^3}{q^2L}
   \det\!\left((\mathsf B_\ell)_{\{1,3\}}\right).
 \end{aligned}
\end{equation}
Consequently,
\begin{equation}\label{main:2dlower}
 q^2\beta^2\max\{|\partial_a^2\ph\,\partial_b^2\ph-(\partial_a\partial_b\ph)^2|,
 |\partial_a^2\ph\,\partial_s^2\ph-(\partial_a\partial_s\ph)^2|\}
 \asymp_LF_\ell^2.
\end{equation}

Fix a sufficiently large constant $K_*$, and call a difference
\emph{$a$-dominant} if
\begin{equation}\label{main:dominant}
 |\ell_1|\ge L^{K_*}N^{11/10}/(DL).
\end{equation}
For such a difference, $n_0,h_0\ll N$ and $d_0\asymp D$ give
\[
 \frac{n_0|\ell_2|+h_0|\ell_3|}{d_0|\ell_1|}
 \ll L^{1-K_*}.
\]
Consequently, \eqref{main:tau} implies
\[
 \tau_1=\frac{z_0\ell_1}{q}\bigl(1+O(L^{1-K_*})\bigr),\qquad
 \frac{|\tau_2|+|\tau_3|}{|\tau_1|}
 \ll L^{1-K_*}\sqrt{\frac{v_0z_0}{Lh_0}}=o(1).
\]
In particular, the transverse terms cannot cancel the main term in
$\tau_1$, and $F_\ell\asymp D|\ell_1|$.
Using
\[
 \det V_0=\frac{q^3Ld_0^2}{n_0v_0z_0^4},\qquad
 \frac{|\ps(x_0)|^3}{(\det V_0)^2}
 =\frac{n_0^2v_0^2z_0^5}{q^3L^2d_0},
\]
we obtain
\begin{equation}\label{main:3dlower}
 |\det\nabla^2\ph|\asymp_L\frac{N^4|\ell_1|^3}{q^6D},
 \qquad |\det(S\nabla^2\ph S)|\asymp_LF_\ell^3.
\end{equation}
Section~\ref{fourth:appendix} also gives
\begin{equation}\label{main:purefourth}
 q^4\partial_a^4\ps=-qd z^3(1+O(L^{-2})).
\end{equation}
Integrating along the shift and using \eqref{main:dominant} to absorb
the two transverse terms yields
\begin{equation}\label{main:thirdlower}
 q^3|\partial_a^3\ph|\asymp_LF_\ell.
\end{equation}
Since $x_0$ was arbitrary, these bounds hold uniformly for
$\inf_{x'\in\supp W}|S^{-1}(x-x')|<\tfrac12L^{-B_*}$,
with \eqref{main:3dlower} and \eqref{main:thirdlower} restricted to
$a$-dominant shifts.

\section{The three-dimensional transformation and the third derivative estimate}
\label{sec:convex-finish}

We first construct a partition that will also be used in
\S\ref{sec:remaining-shifts}.

\subsection{Localization}\label{sec:partition}
Put $\rho=N^{-1/4}$. Choose a fixed nonnegative
$\eta\in C_c^\infty(\RR^3)$ with $\eta(y)=1$ for $|y|\le\sqrt3/2$
and $\supp\eta\subset\{|y|<1\}$. Define
\[
 X(y)=\frac{\eta(y)}{\sum_{m\in\ZZ^3}\eta(y-m)}.
\]
The denominator is locally finite, smooth, periodic, and at least one.
Thus $X$ is fixed independently of all parameters, and
$\sum_{j\in\ZZ^3}X(y-j)=1$. Set
\begin{equation}\label{main:partition}
 x_j=\rho Sj,\qquad
 \chi_j(x)=X(\rho^{-1}S^{-1}(x-x_j)),\qquad W_{\ell,j}=\chi_jW_\ell.
\end{equation}
At most $27$ supports $\supp\chi_j$ meet at any point. Retain the
indices $j$ for which $\supp\chi_j\cap\supp W\ne\varnothing$.
This gives $O_L(\rho^{-3})$ indices, independent of
$\ell$, and
\[
 C_\ell=\sum_jC_{\ell,j},\qquad
 C_{\ell,j}=\sum_{x\in\ZZ^3}W_{\ell,j}(x)\ee(\ph(x)).
\]
For every fixed derivative order,
\begin{equation}\label{main:cutoffbounds}
 q^i\beta^{j+k}|\partial_a^i\partial_b^j\partial_s^kW_{\ell,m}|
 \ll_L\rho^{-(i+j+k)}.
\end{equation}
Let $Q=[-1,1]^3$, $Q_j=x_j+\rho SQ$, and $4Q_j=x_j+4\rho SQ$.
The localized amplitude $W_{\ell,j}$ satisfies
\[
 \supp W_{\ell,j}\subseteq\supp\chi_j\cap\supp W_\ell
 \subset\{x:|S^{-1}(x-x_j)|<\rho\}
 \subset\operatorname{int}Q_j.
\]
The box $4Q_j$ is the dilation of $Q_j$ about $x_j$ by a factor of
four and provides a neighborhood of $\supp W_{\ell,j}$.
Since $\rho+|S^{-1}\ell|=o(L^{-B})$ for every fixed $B$,
each retained index $j$ satisfies
\[
 \inf_{x'\in\supp W}|S^{-1}(x-x')|<\tfrac12L^{-B_*}
 \qquad(x\in4Q_j).
\]
Every segment $x+t\ell$, $x\in4Q_j$, $0\le t\le1$, lies in the
neighborhood \eqref{first:legal} of $\supp W$.

For a dominant difference, \eqref{main:hessian} and
\eqref{main:3dlower} give the inverse Hessian bound below, while
the third derivative bounds in \eqref{main:shiftbounds} give its
variation:
\[
 \|(S\nabla^2\ph(x_j)S)^{-1}\|\ll_LF_\ell^{-1},\qquad
 \sup_{x\in4Q_j}\|S\nabla^2\ph(x)S-S\nabla^2\ph(x_j)S\|
 \ll_LF_\ell\rho.
\]
For $x,x'\in4Q_j$, convexity keeps the segment
$x'+t(x-x')$, $0\le t\le1$, inside $4Q_j$.
The fundamental theorem of calculus gives
\[
 \begin{aligned}
 &\nabla\ph(x)-\nabla\ph(x')-\nabla^2\ph(x_j)(x-x')\\
 &\quad=\int_0^1
 \bigl[\nabla^2\ph(x'+t(x-x'))-\nabla^2\ph(x_j)\bigr]
 (x-x')\,\dd t.
 \end{aligned}
\]
Multiply by $S^{-1}(\nabla^2\ph(x_j))^{-1}$ and insert
$SS^{-1}$ before $x-x'$. The product of the preceding inverse
and variation bounds is $O_L(\rho)=o(1)$. Since the Euclidean
and maximum norms are equivalent in dimension three, for sufficiently
large $N$ this yields
\begin{equation}\label{eq:local-gradient-comparison}
 \left|S^{-1}\left[
 (\nabla^2\ph(x_j))^{-1}(\nabla\ph(x)-\nabla\ph(x'))
 -(x-x')\right]\right|_\infty
 \le\tfrac14|S^{-1}(x-x')|_\infty.
\end{equation}
If $\nabla\ph(x)=\nabla\ph(x')$, this inequality forces $x=x'$.
Thus $\nabla\ph$ is injective on $4Q_j$.

\subsection{\texorpdfstring{Bounding $C_\ell$ for $a$-dominant shifts}{Bounding the correlation sum for a-dominant shifts}}

\begin{proposition}\label{prop:main-correlations}
For every $a$-dominant shift $\ell$,
\begin{equation}\label{eq:convex-final-global}
 |C_\ell|\ll_L q^{1/2}F_\ell^{7/6}
                  +\rho^{-1/2}F_\ell^{4/3}.
\end{equation}
\end{proposition}

\begin{proof}
Fix a retained index $j$ and put $F=F_\ell$.
We work in the original coordinates $x=(a,b,s)$ on $4Q_j$.
The preceding estimates give
\begin{equation}\label{eq:convex-inputs}
 \begin{gathered}
 q^{\alpha_1}\beta^{\alpha_2+\alpha_3}
 \|\partial^\alpha\ph\|_{L^\infty(4Q_j)}
 \ll_L F\quad(2\le|\alpha|\le12),\\
 \|(S\nabla^2\ph(x_j)S)^{-1}\|\ll_L F^{-1},
 \qquad q^3|\partial_a^3\ph(x_j)|\gg_L F,\\
 \rho^{|\alpha|}q^{\alpha_1}\beta^{\alpha_2+\alpha_3}
 \|\partial^\alpha W_{\ell,j}\|_{L^\infty(4Q_j)}\ll_L1
  \quad(|\alpha|\le16).
 \end{gathered}
\end{equation}
Here $\partial^\alpha=\partial_a^{\alpha_1}
\partial_b^{\alpha_2}\partial_s^{\alpha_3}$.
The third derivative lower bound follows from \eqref{main:thirdlower}.
We shall use
\begin{equation}\label{eq:convex-large-scales}
 F\rho/q\gg_L N^{11/60-\varepsilon},\qquad
 F\rho^2\gg_L N^{3/5},\qquad \beta\ll q.
\end{equation}

For the third derivative test below, we need a convex frequency
region on a neighborhood of which $\nabla\ph$ has a smooth inverse.
Set $k_j=\nabla\ph(x_j)$ and
\begin{equation}\label{eq:convex-Pi}
 \Pi_j=k_j+\rho\nabla^2\ph(x_j)S(2Q).
\end{equation}
The quantitative comparison \eqref{eq:local-gradient-comparison}
provides both the required containment and the inverse map.
Taking $x'=x_j$ and $x\in Q_j$, we obtain
\[
 |S^{-1}(\nabla^2\ph(x_j))^{-1}(\nabla\ph(x)-k_j)|_\infty
 \le5\rho/4<2\rho,
\]
so $\nabla\ph(Q_j)\subset\operatorname{int}\Pi_j$.
For $k\in\Pi_j$, the same two-point estimate shows that
\[
 x\longmapsto x-(\nabla^2\ph(x_j))^{-1}(\nabla\ph(x)-k)
\]
is a contraction with constant $1/4$ in the distance
$|S^{-1}(x-x')|_\infty$. Its value at $x_j$ is within $2\rho$
of $x_j$ in this distance, so it maps $4Q_j$ into
$x_j+3\rho SQ\subset\operatorname{int}4Q_j$.
The contraction mapping theorem gives a solution of $\nabla\ph(x)=k$.
By the injectivity proved above and the nondegeneracy of the Hessian,
the inverse $x(k)$ is unique and smooth on the open set
$\nabla\ph(\operatorname{int}4Q_j)$, which contains $\Pi_j$.
Thus
\begin{equation}\label{eq:convex-sandwich}
 \nabla\ph(Q_j)\subset\operatorname{int}\Pi_j,
 \qquad
 \Pi_j\subset\nabla\ph(\operatorname{int}4Q_j).
\end{equation}
On this open set, define
\begin{equation}\label{eq:convex-dual}
 \Theta_j(k)=\ph(x(k))-k\cdot x(k),\qquad
 \widetilde W_{\ell,j}(k)=
 \frac{W_{\ell,j}(x(k))\ee(\sigma_j/8)}
 {\sqrt{|\det\nabla^2\ph(x(k))|}},
\end{equation}
where $\sigma_j$ is the constant signature of the Hessian on $4Q_j$.
The dual amplitude is compactly supported in
$\operatorname{int}\Pi_j$ and extends smoothly by zero to $\mathbb R^3$.

Apply Proposition~\ref{prop:multidimensional-B} with phase $\ph$,
amplitude $W_{\ell,j}$, and
\[
 m=3,\qquad c=x_j,\qquad \mathsf X=\rho S,\qquad
 T=F\rho^2,\qquad \xi=k_j.
\]
The normalization in that proposition is the single substitution
$x=x_j+\rho Su$. The amplitude is supported in
$x_j+\rho S\operatorname{int}Q$, with a fixed normalized margin
inside $4Q_j$.
The derivative bounds in \eqref{eq:convex-inputs}, together with the
Hessian variation bound, give the required normalized derivative
and inverse Hessian bounds on $4Q$.
Injectivity was proved in \S\ref{sec:partition}.
The centered gradient condition follows from
\eqref{eq:local-gradient-comparison} and \eqref{main:hessian}:
\[
 |\rho S(\nabla\ph(x)-k_j)|_\infty\ll_L F\rho^2=T
 \qquad(x\in Q_j).
\]
By \eqref{eq:convex-sandwich}, the transformed amplitude vanishes
outside $\Pi_j$. Hence \eqref{eq:general-B-formula} gives
\begin{equation}\label{eq:convex-remainder}
 C_{\ell,j}=
 \sum_{k\in\Pi_j\cap\mathbb Z^3}
       \widetilde W_{\ell,j}(k)\ee(\Theta_j(k))
       +\mathcal E_{\ell,j}.
\end{equation}
The general error bound \eqref{eq:general-B-error} gives
\begin{align}
 |\mathcal E_{\ell,j}|
 &\ll_L q\beta^2\rho^3T^{-5/2}
       (1+F\rho/q)(1+F\rho/\beta)^2
       \label{eq:secondB-general-error}\\
 &\ll_L\rho F^{1/2}.
       \label{eq:secondB-error}
\end{align}
Indeed, \eqref{eq:convex-large-scales} implies that both
$F\rho/q$ and $F\rho/\beta$ exceed one for sufficiently large $N$.
Thus the frequency factor is $O_L(F^3\rho^3/(q\beta^2))$;
multiplying by $q\beta^2\rho^3(F\rho^2)^{-5/2}$ gives
$\rho F^{1/2}$.

Let $J=\lceil N/(DL)\rceil$, so that $J\beta\asymp q$, and set
\[
 \mathcal V=\{(1,Jr,Js):r,s\in\{0,1,2,3\}\}.
\]
Differentiating $\nabla\ph(x(k))=k$ and the definition of
$\Theta_j$ gives
\begin{equation}\label{eq:convex-third-identity}
 \begin{aligned}
 \partial_\nu x(k)&=(\nabla^2\ph(x(k)))^{-1}\nu,\\
 \partial_\nu^3\Theta_j(k)
 &=\nabla^3\ph(x(k))
       [(\nabla^2\ph(x(k)))^{-1}\nu]^3.
 \end{aligned}
\end{equation}
The notation $B[v]^3$ means $B[v,v,v]$ for a trilinear form $B$.
The homogeneous cubic
\[
 \mathcal P_j(u)=\nabla^3\ph(x_j)
       [(\nabla^2\ph(x_j))^{-1}S^{-1}u]^3
\]
has largest coefficient of size $\asymp_L F^{-2}$.
The upper bound follows from the weighted third derivative bounds
and the normalized inverse Hessian bound in \eqref{eq:convex-inputs}.
For the lower bound, evaluate at
$u=qS\nabla^2\ph(x_j)(1,0,0)^T$.
This vector has norm $O_L(F)$, and
$\mathcal P_j(u)=q^3\partial_a^3\ph(x_j)$ has absolute value $\gg_LF$.
Since $\mathcal P_j$ is homogeneous of degree three, at least one coefficient
has absolute value $\gg_LF^{-2}$.
For $\nu=(1,Jr,Js)\in\mathcal V$, homogeneity gives
\[
 \partial_\nu^3\Theta_j(k_j)
 =\mathcal P_j(S\nu)=q^3\mathcal P_j(1,J\beta r/q,J\beta s/q).
\]
The restriction $\mathcal P_j(1,u_2,u_3)$ retains every coefficient:
homogeneity gives a one-to-one correspondence between
$u_1^{3-a-b}u_2^au_3^b$ and $u_2^au_3^b$ for $a,b\ge0$, $a+b\le3$.
Interpolation in each variable at the four points
$0,J\beta/q,2J\beta/q,3J\beta/q$ gives
\[
 \max_{|\alpha|=3}|[u^\alpha]\mathcal P_j|
 \ll\max_{r,s\in\{0,1,2,3\}}
       |\mathcal P_j(1,J\beta r/q,J\beta s/q)|.
\]
Here $[u^\alpha]\mathcal P_j$ denotes the coefficient of $u^\alpha$.
The constant is uniform since $1\le J\beta/q\le1+C$.
Thus one $\nu_j\in\mathcal V$ satisfies
\[
 |\partial_{\nu_j}^3\Theta_j(k_j)|\asymp_L q^3F^{-2}.
\]

For $x\in4Q_j$, the mean value theorem and the inverse matrix
identity give
\[
 \begin{gathered}
 \max_{|\alpha|=3}q^{\alpha_1}\beta^{\alpha_2+\alpha_3}
 |\partial^\alpha\ph(x)-\partial^\alpha\ph(x_j)|
 \ll_L\rho F,\\
 \|(S\nabla^2\ph(x)S)^{-1}
       -(S\nabla^2\ph(x_j)S)^{-1}\|
 \ll_L\rho F^{-1}.
 \end{gathered}
\]
Thus the derivative in \eqref{eq:convex-third-identity} varies by
$O_L(\rho q^3F^{-2})$. For sufficiently large $N$, this is at most
half its absolute value at $k_j$, since $\rho L^B\to0$ for every
fixed $B$. Put
\[
 \lambda_j=\tfrac12|\partial_{\nu_j}^3\Theta_j(k_j)|.
\]
The chosen direction therefore satisfies throughout $\Pi_j$
\begin{equation}\label{eq:convex-direction-lower}
 \lambda_j\le |\partial_{\nu_j}^3\Theta_j(k)|\le3\lambda_j,
 \qquad \lambda_j\asymp_L q^3F^{-2}.
\end{equation}
Write $\nu_j=(1,Jr,Js)$. Each $k\in\mathbb Z^3$ has the unique form
\[
 k=(t,m_2+Jrt,m_3+Jst),\qquad
 t=k_1,\quad m_2=k_2-Jrk_1,\quad m_3=k_3-Jsk_1.
\]
For fixed $m_2,m_3\in\mathbb Z$, let
\[
 I_{m_2,m_3}
 =\{t\in\mathbb R:(t,m_2+Jrt,m_3+Jst)\in\Pi_j\}.
\]
Since $\Pi_j$ is a convex parallelepiped, each nonempty set
$I_{m_2,m_3}$ is a closed interval, possibly a singleton.
These lines partition $\Pi_j\cap\mathbb Z^3$.
By \eqref{eq:convex-Pi},
\[
 \Pi_j\subset k_j+\bigl(
 [-\mathcal X,\mathcal X]\times[-J\mathcal X,J\mathcal X]^2\bigr),
 \qquad \mathcal X\asymp_L F\rho/q\ge1.
\]
Thus
\begin{equation}\label{eq:convex-lines-count}
 \#\{(m_2,m_3):I_{m_2,m_3}\ne\varnothing\}\ll J^2\mathcal X^2,
 \qquad |I_{m_2,m_3}|\le2\mathcal X.
\end{equation}

Using $\det S=q\beta^2$, we have
\[
 \widetilde W_{\ell,j}(k)
 =q\beta^2 W_{\ell,j}(x(k))
        |\det(S\nabla^2\ph(x(k))S)|^{-1/2}\ee(\sigma_j/8).
\]
Jacobi's identity
\[
 \dd\log|\det(S\nabla^2\ph S)|
 =\operatorname{tr}\!\left((S\nabla^2\ph S)^{-1}\dd(S\nabla^2\ph S)\right)
\]
and \eqref{eq:convex-inputs}, \eqref{main:3dlower} show that the
determinant factor and its derivatives under
$q\partial_a$, $\beta\partial_b$, and $\beta\partial_s$ are
$O_L(F^{-3/2})$ on $4Q_j$.
The inverse-map identity and the chain rule give
\[
 \begin{aligned}
 |S^{-1}\partial_{\nu_j}x(k)|
 &=|(S\nabla^2\ph(x(k))S)^{-1}S\nu_j|\ll_L q/F,\\
 |\partial_{\nu_j}(W_{\ell,j}(x(k)))|
 &\ll_L\rho^{-1}|S^{-1}\partial_{\nu_j}x(k)|
 \ll_L q/(F\rho).
 \end{aligned}
\]
The product rule therefore yields
\[
 \|\widetilde W_{\ell,j}\|_{L^\infty(\Pi_j)}
 +\frac{F\rho}{q}\|\partial_{\nu_j}\widetilde W_{\ell,j}\|_{L^\infty(\Pi_j)}
 \ll_L q\beta^2 F^{-3/2}.
\]
For $t\in I_{m_2,m_3}$, define
\[
 b_{m_2,m_3}(t)=
 \widetilde W_{\ell,j}(t,m_2+Jrt,m_3+Jst).
\]
This restriction satisfies
\begin{equation}\label{eq:convex-variation}
 \|b_{m_2,m_3}\|_{L^\infty(I_{m_2,m_3})}+
 \int_{I_{m_2,m_3}}|b'_{m_2,m_3}(t)|\,dt
 \ll_L q\beta^2F^{-3/2}.
\end{equation}
The third derivative test (see Robert~\cite[Section~3.3]{Robert}) gives the bound
\[
 O\bigl(\mathcal X\lambda_j^{1/6}+\mathcal X^{1/2}\lambda_j^{-1/6}\bigr)
\]
on the integer points of every subinterval of $I_{m_2,m_3}$.
Abel summation, followed by
\eqref{eq:convex-lines-count}, now yields
\begin{align}\label{eq:convex-main-bound}
 \left|\sum_{k\in\Pi_j\cap\mathbb Z^3}
       \widetilde W_{\ell,j}(k)\ee(\Theta_j(k))\right|
 &\ll_L q\beta^2 F^{-3/2}J^2\mathcal X^2
     \bigl(\mathcal X\lambda_j^{1/6}+\mathcal X^{1/2}\lambda_j^{-1/6}\bigr)
     \notag\\
 &\ll_L\rho^3q^{1/2}F^{7/6}+\rho^{5/2}F^{4/3}.
\end{align}
Since $F\rho^2\ge1$, the error in \eqref{eq:secondB-error} is
absorbed by the second term. Hence
\begin{equation}\label{eq:convex-final-local}
 |C_{\ell,j}|\ll_L
       \rho^3q^{1/2}F_\ell^{7/6}+\rho^{5/2}F_\ell^{4/3}.
\end{equation}
There are $O_L(\rho^{-3})$ retained indices $j$ in
\eqref{main:partition}. Summing $C_{\ell,j}$ over these indices proves
\eqref{eq:convex-final-global}.
\end{proof}

\section{The remaining shifts}\label{sec:remaining-shifts}

\subsection{A two-dimensional estimate}\label{sec:two-dimensional}

We use the partition from \S\ref{sec:partition}.
For every nonzero shift $\ell$ and each retained box $Q_j$, choose
at $x_j$ the $(a,b)$ or $(a,s)$ principal Hessian block whose
determinant satisfies \eqref{main:2dlower}.
By \eqref{main:shiftbounds}, the same block remains invertible
throughout $4Q_j$: its normalized inverse is $O_L(F_\ell^{-1})$,
and its relative variation is $O_L(\rho)$.
For the $(a,b)$ block, fix $s=s_0\in\mathbb Z$ with
$|s_0-(x_j)_3|\le\rho\beta$; for the $(a,s)$ block, fix
$b=b_0\in\mathbb Z$ with $|b_0-(x_j)_2|\le\rho\beta$.
The two varying coordinates have scales $q,\beta$, and the slice
centre has these coordinates equal to those of $x_j$.
The segment argument in \S\ref{sec:partition} shows that the
restricted gradient is injective on the corresponding slice of
$4Q_j$, uniformly in the fixed coordinate.

Apply Proposition~\ref{prop:multidimensional-B}, in the form
\eqref{eq:general-B-bound}, to each slice with
\[
 m=2,\qquad \mathsf X=\diag(\rho q,\rho\beta),
 \qquad T=F_\ell\rho^2.
\]
The frequency centre is the restricted gradient of $\ph$ at the
slice centre. After rescaling, the amplitude has $C^{16}$ norm
$O_L(1)$ and is supported in $\operatorname{int}[-1,1]^2$.
The phase divided by $T$ has derivatives of orders $2$ through $12$
and inverse Hessian $O_L(1)$ on $4[-1,1]^2$.
These bounds follow from \eqref{main:shiftbounds} and
\eqref{main:cutoffbounds}; the first derivative bound relative to
the frequency centre follows by integrating the Hessian.
The parameter bounds and $F_\ell\gg D$ give
\[
 \rho\beta>N^{1/4}L,\qquad
 \frac{F_\ell\rho}{\beta}\gg
       \frac{N^{1/12-\varepsilon}}L,\qquad
 T\gg N^{1/3-\varepsilon}.
\]
There are $O(1+\rho\beta)$ choices of the fixed coordinate.
Summing the bound for the slices gives
\begin{equation}\label{eq:secondB-2d-entry}
 |C_{\ell,j}|\ll_L
 (1+\rho\beta)\frac{q\beta}{F_\ell}
 (1+F_\ell\rho/q)(1+F_\ell\rho/\beta).
\end{equation}
Using $\rho\beta\gg1$ and $F_\ell\rho/\beta\gg1$, while retaining
the factor $1+F_\ell\rho/q$, we obtain
\begin{equation}\label{eq:secondB-2d-simplified}
 |C_{\ell,j}|\ll_L q\beta\rho^2+\beta F_\ell\rho^3.
\end{equation}

\subsection{The average over shifts}

Let $\mathcal E$ consist of the nonzero shifts in
$\mathcal R-\mathcal R$ for which
\[
 |\ell_1|<L^{K_*}\frac{N^{11/10}}{DL}.
\]
Fix
\[
 M=C L^{K_*+1}N^{11/10},\qquad \delta=\frac{M}{\beta^2},
\]
where $C$ is chosen so that $F_\ell\le M$ for $\ell\in\mathcal E$.
The partition in \S\ref{sec:partition} has $O_L(\rho^{-3})$ centres,
independent of $\ell$.

\begin{lemma}\label{lem:remaining-shift-average}
We have
\begin{equation}\label{eq:remaining-shift-average}
 \frac1{\#\mathcal R}\sum_{\ell\in\mathcal E}|C_\ell|
 \ll_L D\beta N^{-1/25}+\beta^{1/3}N^{107/75}.
\end{equation}
\end{lemma}

\begin{proof}
Fix a retained index $j$ in \eqref{main:partition}.
At the centre $x_j$ of $Q_j$, define the homogeneous cubic polynomial
\begin{equation}\label{eq:remaining-cubic}
 P_j(\ell)=M^{-3}\det\!\left(
 S\bigl[\ell_1\partial_a\nabla^2\bar\Psi
       +\ell_2\partial_b\nabla^2\bar\Psi
       +\ell_3\partial_s\nabla^2\bar\Psi\bigr](x_j)S
 \right).
\end{equation}
Take $x_0=x_j$ in the fixed frame. Since
$\partial_1=(q/z_0)\partial_a$ and
$\bar\Psi(x_0)=-qd_0/z_0$, we have
\[
 S(\partial_a\nabla^2\bar\Psi)(x_j)S
 =-d_0\,S V_0^{-T}T_1V_0^{-1}S.
\]
Equations \eqref{main:Tmatrices}--\eqref{main:combination3} give
$|\det T_1|\gg1$, while $|\det S/\det V_0|\asymp_L1$.
Hence
\begin{equation}\label{eq:remaining-leading-coefficient}
 |[\ell_1^3]P_j|
 =\frac{d_0^3}{M^3}
 \left|\frac{\det S}{\det V_0}\right|^2|\det T_1|
 \gg_L(D/M)^3.
\end{equation}

Put $p=|P_j(\ell)|$. We shall prove
\begin{equation}\label{eq:remaining-capped-bound}
 |C_{\ell,j}|\ll_L\rho^3M^{3/2}
 \min\left\{\frac{\beta}{\sqrt M},\ p^{-1/2}\right\}.
\end{equation}
When $p=0$, the minimum denotes its first argument.
The threshold $\delta=M/\beta^2$ is where these two bounds agree,
since $\delta^{-1/2}=\beta/\sqrt M$.
For every $\ell\in\mathcal E$, the two-dimensional estimate gives
\[
 |C_{\ell,j}|\ll_Lq\beta\rho^2+\beta F_\ell\rho^3
 \ll_L\rho^3\beta M,
\]
because $F_\ell\le M$ and $M\rho/q\gg_LN^{11/60-\varepsilon}$.
This proves \eqref{eq:remaining-capped-bound} when $p<\delta$.

Suppose now that $p\ge\delta$.
The fundamental theorem of calculus gives, for $x\in4Q_j$,
\[
 S\nabla^2\ph(x)S
 =\int_0^1
 S\bigl[(\ell\cdot\nabla)\nabla^2\bar\Psi\bigr](x+t\ell)S
 \,\dd t.
\]
Freezing the integrand at $x_j$, the fourth derivative bounds in
\eqref{main:inputbounds} show that
\[
 \left\|S\nabla^2\ph(x)S
 -S\bigl[(\ell\cdot\nabla)\nabla^2\bar\Psi\bigr](x_j)S\right\|
 \ll_L F_\ell\bigl(\rho+|S^{-1}\ell|\bigr).
\]
Both matrices have norm $O_L(F_\ell)$. Multilinearity of the
determinant and $F_\ell\le M$ therefore give
\begin{equation}\label{eq:remaining-freezing}
 \left|M^{-3}\det(S\nabla^2\ph(x)S)-P_j(\ell)\right|
 \ll_L\rho+|S^{-1}\ell|.
\end{equation}
Since $\beta\ll_LN^{2/3+\varepsilon}$,
\begin{equation}\label{eq:remaining-freezing-margin}
 \delta\gg_L N^{-7/30-2\varepsilon},\qquad
 \frac{\rho+|S^{-1}\ell|}{\delta}
 \ll_LN^{-1/60+2\varepsilon}=o(1).
\end{equation}
Thus, uniformly on $4Q_j$,
\begin{equation}\label{eq:remaining-det-comparison}
 |\det(S\nabla^2\ph(x)S)|\asymp M^3p,
 \qquad
 \|(S\nabla^2\ph(x)S)^{-1}\|\ll_L(Mp)^{-1}.
\end{equation}
Here the inverse bound follows from the cofactor formula and
$\|S\nabla^2\ph(x)S\|\ll_LM$.

We apply the near-constant-Hessian bound
\eqref{eq:general-B-nearconstant} of
Proposition~\ref{prop:multidimensional-B} with
\[
 m=3,\qquad c=x_j,\qquad \mathsf X=\rho S,\qquad T=M\rho^2.
\]
We verify its relative Hessian condition in the original coordinates.
The preceding inverse bound and \eqref{main:shiftbounds} give
\[
 \sup_{x\in4Q_j}
 \left\|(S\nabla^2\ph(x_j)S)^{-1}
 S\bigl(\nabla^2\ph(x)-\nabla^2\ph(x_j)\bigr)S\right\|
 \ll_L\rho/p.
\]
For $1\le|\alpha|\le9$, they also give
\[
 \begin{aligned}
 &\rho^{|\alpha|}q^{\alpha_1}\beta^{\alpha_2+\alpha_3}
 \sup_{x\in4Q_j}
 \left\|(S\nabla^2\ph(x_j)S)^{-1}
 S(\partial^\alpha\nabla^2\ph)(x)S\right\|\\
 &\hspace{35mm}\ll_L\rho^{|\alpha|}/p.
 \end{aligned}
\]
These are precisely the derivatives of the relative Hessian after
$x=x_j+\rho Su$. Its $C^9$ norm is therefore
$O_L(\rho/p)=o(1)$ by \eqref{eq:remaining-freezing-margin}.
The rescaled Hessian has determinant $\asymp T^3p$ and inverse
$O_L((Tp)^{-1})$; its higher phase derivatives are $O_L(T)$.
The rescaled amplitude has $C^{16}$ norm $O_L(1)$ by
\eqref{main:cutoffbounds}. The first derivative bound relative to
the frequency centre $\nabla\ph(x_j)$ follows by integrating
\eqref{main:hessian}.
Since $M\rho/q\gg_LN^{11/60-\varepsilon}$ and $\beta\ll q$,
the bound \eqref{eq:general-B-nearconstant} yields
\[
 \begin{aligned}
 |C_{\ell,j}|
 &\ll_L q\beta^2\rho^3
 \bigl(T^{-3/2}p^{-1/2}+(Tp^2)^{-8}\bigr)
 (1+M\rho/q)(1+M\rho/\beta)^2\\
 &\ll_L\rho^3M^{3/2}p^{-1/2}+T^3(Tp^2)^{-8}.
 \end{aligned}
\]
The entries of the matrix defining $P_j$ are $O_L(M)$, so
$p\ll_L1$. The last error is absorbed by the first term, because
\[
 \begin{gathered}
 Tp^2\ge M\rho^2\delta^2
   =\frac{\rho^2M^3}{\beta^4}
   \gg_LN^{2/15-4\varepsilon},\\
 T^{3/2}\left(\frac{\rho^2M^3}{\beta^4}\right)^{-8}
   \ll_LN^{-1/6+32\varepsilon}=o(1).
 \end{gathered}
\]
This proves \eqref{eq:remaining-capped-bound} also for $p\ge\delta$.

For a cubic polynomial
$P(t)=c\prod_{r=1}^3(t-z_r)$ with $c\ne0$, we have, for every $U>0$,
\begin{equation}\label{eq:remaining-cubic-sum}
 \sum_{k\in\mathbb Z}\min\{U,|P(k)|^{-1/2}\}
 \ll U+|c|^{-1/3}U^{1/3}.
\end{equation}
Indeed, $|t-z_r|\ge|t-\operatorname{Re}z_r|$, and
\[
 |P(t)|^{-1/2}
 \le |c|^{-1/2}\sum_{r=1}^3
       |t-\operatorname{Re}z_r|^{-3/2}.
\]
Taking the minimum with $U$ bounds the summand by the sum of the
three separately truncated terms.
For each real centre, the two nearest integers contribute at most
$2U$. Compare the remaining terms with
\[
 \int_0^\infty\min\{U,|c|^{-1/2}t^{-3/2}\}\,\dd t
 =3|c|^{-1/3}U^{1/3}.
\]
This proves \eqref{eq:remaining-cubic-sum}, including repeated and
nonreal roots.

Apply \eqref{eq:remaining-cubic-sum} in $\ell_1$, with $\ell_2,\ell_3$
fixed and $U=\beta/\sqrt M$.
Equations \eqref{eq:remaining-leading-coefficient} and
\eqref{eq:remaining-capped-bound} give
\[
 \sum_{\ell_1:\,\ell\in\mathcal E}|C_{\ell,j}|
 \ll_L\rho^3M^{3/2}
 \left\{\frac{\beta}{\sqrt M}
       +\frac MD\left(\frac{\beta}{\sqrt M}\right)^{1/3}\right\}.
\]
There are $O(N^{1/5})$ pairs $(\ell_2,\ell_3)$ satisfying
\eqref{main:shifts} and $O_L(\rho^{-3})$ retained indices $j$.
Summing and dividing by $\#\mathcal R\asymp N^{67/50}/(DL)$ gives
\begin{align*}
 \frac1{\#\mathcal R}\sum_{\ell\in\mathcal E}|C_\ell|
 &\ll_L D N^{-57/50}M^{3/2}
 \left\{\frac{\beta}{\sqrt M}
       +\frac MD\left(\frac{\beta}{\sqrt M}\right)^{1/3}\right\}\\
 &\ll_L D\beta N^{-1/25}+\beta^{1/3}N^{107/75},
\end{align*}
as required.
\end{proof}

\subsection{Completion of the counting estimate}\label{sec:recover}
\begin{proof}[Proof of Proposition~\ref{prop:energy}]
We first consider $qD>N^{3/2}$ and apply \eqref{main:Aentry}.
For dominant shifts, $F_\ell\ll N^{57/50}/L$ and
$\#(\mathcal R-\mathcal R)\ll\#\mathcal R$.
Proposition~\ref{prop:main-correlations} gives
\begin{align*}
 N\sqrt D\left(\frac1{\#\mathcal R}
       \sum_{\ell\ \mathrm{dominant}}|C_\ell|\right)^{1/2}
 &\ll_LN\sqrt D\left[
 q^{1/4}N^{(57/50)(7/12)}
       +\rho^{-1/4}N^{(57/50)(2/3)}\right]\\
 &\ll_LN^{1399/600+\varepsilon/4}+N^{929/400}.
\end{align*}
For the remaining nonzero shifts, Lemma~\ref{lem:remaining-shift-average}
gives
\begin{align}\label{eq:remaining-A-return}
 N\sqrt D\left(\frac1{\#\mathcal R}
          \sum_{\ell\in\mathcal E}|C_\ell|\right)^{1/2}
 &\ll_L N^{12/25}q^{1/2}D^{3/2}
       +N^{116/75}q^{1/6}D^{2/3}\notag\\
 &\ll_L N^{347/150+\varepsilon/2}
       +N^{523/225+\varepsilon/6}.
\end{align}
The zero shift contributes at most
\[
 \frac{(qD)^{3/2}}{\sqrt{\#\mathcal R}}
 \ll_Lq^{3/2}D^2N^{-67/100}
 \ll_LN^{233/100+3\varepsilon/2}.
\]
Finally, \eqref{main:firstB}, $qD>N^{3/2}$, and
$q<N^{2/3+\varepsilon}$, $D\ll N/L$ give
\[
 |E_q|\ll_LN^{-1/4}+N^{5/6+\varepsilon/2}.
\]
All these exponents are less than $7/3-\varepsilon$ for
$\varepsilon=10^{-4}$. Together with \eqref{main:small}, which handles
$qD\le N^{3/2}$, this proves
\[
 |S_q|\ll_LN^{7/3-\varepsilon}\qquad(1\le q<R)
\]
uniformly for $P,H,D,q$ satisfying \eqref{main:range}.
Insert this in \eqref{first:Fejer} and sum over the $O(L)$ dyadic
values of $D$ allowed by \eqref{main:range}.
The zero frequency satisfies \eqref{first:zero}.
The terms with $d=0$ contribute $O(N^2)$, and negative $d$ are
handled by symmetry. Thus
\[
 \mathcal E_{P,H}(N^{-1})\ll_LN^{7/3-\varepsilon}.
\]
Absorbing the fixed logarithmic factor proves \eqref{eq:energy},
for example with $\sigma=\varepsilon/2$.
\end{proof}

\appendix

\section{The root equation and the first stationary map}\label{app:root}

Throughout this appendix, $N/(2L^4)\le P,H\le N$.
Constants may depend on the fixed cutoff functions
$\theta_d,\theta_n,\theta_h$, the constant $C_\Omega$ defining
$\Omega_{D,P,H}$, and the derivative order.
The exponent $B$ may vary between estimates.

\subsection{Root identities and uniform estimates}\label{app:root-range}

For $n>1$ and $0\le d<h$, let $\kappa=\kappa(d,n,h)$ solve
\begin{equation}\label{app:root-equation}
 (\kappa+h-d)\log(\kappa+h-d)-\kappa\log\kappa
 =(n+h)\log(n+h)-n\log n.
\end{equation}
The left side increases strictly to infinity and at $\kappa=n+d$ does not exceed the right side.
Hence there is a unique root $\kappa\ge n+d$, with strict inequality for $d>0$.
Its smoothness follows from the implicit function theorem.
Write
\[
 v=n+h,\qquad w=\kappa+h-d,\qquad z=\log(w/\kappa).
\]
If $d\asymp D$, $n,\kappa\asymp P$ and $h\asymp H$, then
\begin{equation}\label{app:D-range}
 D\ll \frac{P\log(1+H/P)}L\le\frac HL.
\end{equation}
The root equation gives
\begin{equation}\label{app:root-integral}
 \int_n^\kappa\log\left(1+\frac{h-d}{x}\right)\,\dd x
 =\int_{h-d}^h[1+\log(n+t)]\,\dd t.
\end{equation}
The right side is $\gg dL$, whereas the left side is $\ll P\log(1+H/P)$.

Assume now that $D$ satisfies the upper bound in \eqref{main:range},
and fix $C_0>1$. For sufficiently large $N$, the region
\begin{equation}\label{app:source-enlarged}
 0\le d\le C_0D,\qquad P/C_0\le n\le C_0P,
 \qquad H/C_0\le h\le C_0H
\end{equation}
satisfies $d/h\ll L^{-1}$ and
\begin{equation}\label{app:root-scales}
 \kappa\asymp n\asymp P,\qquad w\asymp v,\qquad
 z\asymp\log(1+h/n),\qquad L^{-4}\ll z\ll\log L.
\end{equation}
For $x\ge n$, concavity of the logarithm gives
\[
 \log\left(1+\frac{h-d}{x}\right)
 \ge\frac nx\log\left(1+\frac{h-d}{n}\right).
\]
Hence \eqref{app:root-integral} implies
\[
 n\log\left(1+\frac{h-d}{n}\right)\log\frac\kappa n
 \ll dL\ll P\log(1+H/P).
\]
The first factor on the left is comparable to $P\log(1+H/P)$.
Thus $\kappa\asymp n$, which gives the other estimates in \eqref{app:root-scales}.

If $n,m\asymp P$, $h\asymp H$, $1\le d<h$ and
$|G_h(n)-G_{h-d}(m)|\le N^{-1}$, then
\begin{equation}\label{app:near-integer}
 d\ll\frac{P\log(1+H/P)}L\ll\frac HL,
 \qquad |m-\kappa(d,n,h)|\ll\frac{L^4}{N}.
\end{equation}
Replacing $\kappa$ by $m$ in \eqref{app:root-integral} introduces an error of at most $N^{-1}$ and proves the bound for $d$.
Equation~\eqref{app:root-scales} then gives $\kappa\asymp P$.
On the segment joining $m$ and $\kappa$,
$G_{h-d}'(x)=\log(1+(h-d)/x)\gg L^{-4}$.
The mean value theorem proves the second bound.
When $d=0$, the same derivative bound gives
$|G_h(n)-G_h(m)|\gg L^{-4}|n-m|$.
Thus a pair of integers satisfying $|G_h(n)-G_h(m)|\le N^{-1}$ must have $m=n$ for sufficiently large $N$.

Differentiation of \eqref{app:root-equation} gives
\begin{equation}\label{app:root-first}
 \partial_d\kappa=\frac{1+\log w}{z},\qquad
 \partial_n\kappa=\frac{\log(v/n)}z,\qquad
 \partial_h\kappa=\frac{\log(v/w)}z.
\end{equation}
Define
\[
 \Delta=n(1+\log n)^2+w(1+\log w)^2
       -\kappa(1+\log\kappa)^2-v(1+\log v)^2.
\]
Since $\partial_dw=(1+\log\kappa)/z$, direct differentiation yields
\begin{equation}\label{app:Delta-integral}
 \begin{split}
 \partial_d\Delta&=(1+\log\kappa)(1+\log w),\\
 \Delta&=\int_0^d(1+\log\kappa(t,n,h))(1+\log w(t,n,h))\,\dd t
 \asymp dL^2.
 \end{split}
\end{equation}
We used $\kappa(0,n,h)=n$, $w(0,n,h)=v$, and \eqref{app:root-scales} on $0\le t\le d$.

With gradients taken in $(d,n,h)$, a second differentiation gives
\begin{equation}\label{app:root-hessian}
 z\nabla^2\kappa
 =\frac{\nabla\kappa(\nabla\kappa)^T}{\kappa}
 -\frac{\nabla n(\nabla n)^T}{n}
 +\frac{\nabla v(\nabla v)^T}{v}
 -\frac{\nabla w(\nabla w)^T}{w}.
\end{equation}
The four vectors on the right are
\[
 \begin{pmatrix}\partial_d\kappa\\\partial_n\kappa\\\partial_h\kappa\end{pmatrix},\quad
 \begin{pmatrix}0\\1\\0\end{pmatrix},\quad
 \begin{pmatrix}0\\1\\1\end{pmatrix},\quad
 \begin{pmatrix}\partial_d\kappa-1\\\partial_n\kappa\\\partial_h\kappa+1\end{pmatrix}.
\]
Their $123$, $124$, $134$ and $234$ determinants are, respectively,
$(1+\log w)/z$, $(1+\log v)/z$, $(1+\log n)/z$ and $(1+\log\kappa)/z$.
Expanding the determinant in \eqref{app:root-hessian} therefore gives
\begin{equation}\label{app:root-det}
 \det\nabla^2\kappa=-\frac{\Delta}{nv\kappa wz^5}<0\qquad(d>0).
\end{equation}
Moreover,
\[
 z\partial_d^2\kappa
 =\frac{(\partial_d\kappa)^2}{\kappa}
  -\frac{(\partial_d\kappa-1)^2}{w}>0,
\]
because $w>\kappa$ and $\partial_d\kappa>1$.
Together with \eqref{app:root-det}, this gives two positive eigenvalues and one negative eigenvalue.
Thus the signature is $1$.

The following refinements hold uniformly for the triples
$(d,n,h)$ in \eqref{app:source-enlarged} with $d>0$:
\begin{equation}\label{app:Delta-refined}
 \frac\Delta d=(\log n)^2+O(L(1+z)),\qquad
 \left|\frac{d\partial_d\Delta}{\Delta}-1\right|\ll L^{-1},\qquad
 \frac{n|\partial_n\Delta|+h|\partial_h\Delta|}{\Delta}\ll L^{-1}.
\end{equation}
The functions $\kappa(t,n,h)$ and $w(t,n,h)$ increase with $t$, and their logarithms vary by $O(1)$ on $[0,d]$.
The integrand in \eqref{app:Delta-integral} therefore differs from its endpoint value by $O(L)$, and from $(\log n)^2$ by $O(L(1+z))$.
This proves the first two estimates.
By \eqref{app:root-first}, $|\partial_n\kappa|\ll1$ and
\[
 \frac{h|\partial_h\kappa|}{\kappa}
 \le\frac{h(w-v)}{v\kappa z}\ll\frac h{vz}\ll1.
\]
The $n$- and $h$-weighted derivatives of the integrand are consequently $O(L)$.
Differentiation under the integral sign is valid because
$\kappa(t,n,h)$ and $w(t,n,h)$ are smooth on an open set containing
the segment $\{(t,n,h):0\le t\le d\}$, and proves the last estimate.
We also have
\begin{equation}\label{app:elementary-scales}
 \frac dn\ll\frac{1+z}{L},\qquad \frac dh\ll L^{-1},\qquad
 \frac{v\log v-n\log n}{h}=\log n+O(1+z),\qquad
 \frac h{vz}\asymp\frac1{1+z}.
\end{equation}
The third estimate follows by integrating $1+\log x$ from $n$ to $v$.
The last follows from
$x/((1+x)\log(1+x))\asymp(1+\log(1+x))^{-1}$ for $x>0$.

\subsection{Injectivity of the first stationary map}\label{app:first-unique}

For fixed $q>0$ and $(a,b,s)$, consider
\begin{equation}\label{app:stationary}
 q(\partial_d\kappa,\partial_n\kappa,\partial_h\kappa)=(a,b,s),
 \qquad n>1,\quad 0<d<h.
\end{equation}
Every solution satisfies
\begin{equation}\label{app:four-exponentials}
 w=e^{az/q-1},\quad \kappa=e^{(a/q-1)z-1},\quad
 v=e^{(a+s)z/q-1},\quad n=e^{(a+s-b)z/q-1}.
\end{equation}
Since $w>v$ and $\log(v/n)>z$, we have $s<0$ and $b>q$.
Also, $a+s-b>0$ since $n>1$.

Let $F(z)$ be the expression obtained by substituting \eqref{app:four-exponentials} into
$w\log w-\kappa\log\kappa-v\log v+n\log n$.
Then
\[
 F(z)=e^{-1}z^2\left[
 \int_{a/q-1}^{a/q}te^{tz}\,\dd t
 -\int_{(a+s-b)/q}^{(a+s)/q}te^{tz}\,\dd t\right].
\]
After cancellation of any common part of the two intervals,
\begin{align*}
 \frac{e^{1-(a/q-1)z}}{z^2}F(z)
 ={}&\int_{\max\{a/q-1,(a+s)/q\}}^{a/q}
       te^{(t-a/q+1)z}\,\dd t\\
 &-\int_{(a+s-b)/q}^{\min\{a/q-1,(a+s)/q\}}
       te^{(t-a/q+1)z}\,\dd t.
\end{align*}
All endpoints are positive, and both intervals have positive length.
In the first integral $t-a/q+1\ge0$; in the second it is nonpositive.
The displayed expression is therefore strictly increasing in $z$ and has at most one zero.
Equations \eqref{app:four-exponentials}, followed by $h=v-n$ and $d=\kappa+v-n-w$, determine the stationary point uniquely.

By \eqref{app:root-det}, the map $q\nabla\kappa$ on the domain
$n>1$, $0<d<h$ in \eqref{app:stationary} is also a local
diffeomorphism. Its inverse on the image of this domain under
$q\nabla\kappa$ is therefore smooth.
We write this inverse as $(d(a,b,s),n(a,b,s),h(a,b,s))$.

\subsection{Derivative bounds and the frequency neighborhood}\label{app:coarse}

For fixed nonnegative integers $i,j,k$, throughout \eqref{app:source-enlarged} one has
\begin{equation}\label{app:root-coarse}
 D^iP^jH^k\left|\partial_d^i\partial_n^j\partial_h^k(\kappa-n)\right|
 \ll DL^B.
\end{equation}
To see this, implicit differentiation gives, for $r=i+j+k\ge1$,
\[
 |\partial_d^i\partial_n^j\partial_h^k\kappa|\ll N^{1-r}L^B.
\]
Indeed, $n,v,\kappa,w$ lie between $NL^{-B}$ and $NL^B$, and $z\gg L^{-B}$.
At each step the highest derivative of $\kappa$ has coefficient $z$.
All other terms are bounded by induction, using the derivatives of $x\log x$.
If $i\ge1$, multiplication by $D^iP^jH^k$ proves \eqref{app:root-coarse}.
If $i=0$, use
\[
 \kappa(d,n,h)-n=\int_0^d\partial_d\kappa(t,n,h)\,\dd t
\]
and differentiate under the integral sign.

Within \eqref{app:source-enlarged}, restrict now to $d\asymp D>0$
and put
\begin{equation}\label{app:K-definition}
 K_\kappa=D^{-1}\diag(D,P,H)\nabla^2\kappa\diag(D,P,H).
\end{equation}
Its entries and all fixed-order weighted derivatives are $O(L^B)$, and
\[
 \det K_\kappa=-\frac{P^2H^2\Delta}{Dnv\kappa wz^5},\qquad
 |\det K_\kappa|^{-1}\ll L^B.
\]
The cofactor formula and
$\partial_dK_\kappa^{-1}=-K_\kappa^{-1}(\partial_dK_\kappa)K_\kappa^{-1}$
give the same bounds for $K_\kappa^{-1}$ and its fixed-order weighted derivatives.

On the image under $q\nabla\kappa$ of the domain in
\eqref{app:stationary}, define
\[
 \ps(a,b,s)=q\kappa(d,n,h)-ad-bn-sh,
\]
where $(d,n,h)$ is the inverse image of $(a,b,s)$.
The root equation and stationarity give
\begin{equation}\label{app:dual-identities}
 \ps=-\frac{qd}{z},\qquad
 \nabla_{a,b,s}\ps=-(d,n,h)^T,\qquad
 \partial_a\ps=\frac zq\ps,\qquad
 \nabla^2_{a,b,s}\ps=-\frac1q(\nabla^2_{d,n,h}\kappa)^{-1}.
\end{equation}
The first identity follows by multiplying $\ps/q$ by $z$ and using \eqref{app:root-first}.
For the others, differentiate and cancel the inverse-map derivatives by stationarity.
In particular,
\begin{equation}\label{app:dual-det}
 \det\nabla^2\ps=\frac{nv\kappa wz^5}{q^3\Delta}.
\end{equation}
The chain rule gives
\begin{gather}
 \diag(q,qD/P,qD/H)\nabla^2\ps\diag(q,qD/P,qD/H)
 =-qD K_\kappa^{-1},\label{app:weighted-dual-Hessian}\\
 \begin{pmatrix}q\partial_a\\(qD/P)\partial_b\\(qD/H)\partial_s\end{pmatrix}
 =K_\kappa^{-1}
 \begin{pmatrix}D\partial_d\\P\partial_n\\H\partial_h\end{pmatrix}.
 \label{app:weighted-chain}
\end{gather}
In the second identity, the right side acts before composition with the stationary inverse.
Consequently, with $\beta=qDL/N$, for every fixed order
\begin{equation}\label{app:dual-coarse}
 q^i\beta^{j+k}|\partial_a^i\partial_b^j\partial_s^k\ps|
 \ll qDL^B.
\end{equation}
For orders zero and one this follows from \eqref{app:dual-identities}.
For higher orders, apply \eqref{app:weighted-dual-Hessian} and \eqref{app:weighted-chain} repeatedly and use $PL/N,HL/N\le L$.

We now verify that \eqref{app:dual-coarse} holds throughout the
neighborhood of $\supp W$ specified in \eqref{first:legal}.
The box $\Omega_{D,P,H}$ contains $\supp\omega$ with a fixed
positive margin in the $(D,P,H)$ scales.
Fix $(d_*,n_*,h_*)\in\supp\omega$ and put
$(a_*,b_*,s_*)=q\nabla\kappa(d_*,n_*,h_*)$.
Use the centered coordinates
\[
 u=\left(\frac{d-d_*}{D},\frac{n-n_*}{P},\frac{h-h_*}{H}\right),
 \qquad
 y=\left(\frac{a-a_*}{q},
          \frac{b-b_*}{qD/P},\frac{s-s_*}{qD/H}\right).
\]
The derivative of the stationary map from $u$ to $y$ is
$K_\kappa(d_*+Du_1,n_*+Pu_2,h_*+Hu_3)$.
In the next display, abbreviate this matrix by $K_\kappa(u)$.
This matrix, its first $u$-derivatives and its inverse are $O(L^B)$.
For sufficiently large fixed $B_1$, the closed ball
$\{u:|u|\le r\}$, where $r=L^{-B_1}$, corresponds under
$u\mapsto(d_*+Du_1,n_*+Pu_2,h_*+Hu_3)$ to a subset of
$\Omega_{D,P,H}$ and satisfies
\[
 \|K_\kappa(0)^{-1}\|
 \sup_{|u|\le r}\|K_\kappa(u)-K_\kappa(0)\|\le\tfrac12.
\]
After multiplication by $K_\kappa(0)^{-1}$, the inverse equation
is a contraction on $\{u:|u|\le r\}$ whenever the target satisfies
$|y|<r/(2\|K_\kappa(0)^{-1}\|)$.
It therefore defines an inverse for $|y|<L^{-B_2}$ with a sufficiently
large fixed $B_2$, uniformly for $(d_*,n_*,h_*)\in\supp\omega$.
These target balls are centered at $(a_*,b_*,s_*)$ in the
$(q,qD/P,qD/H)$ frequency scales.
Injectivity makes the inverses agree on overlaps of these balls.
Changing the frequency scales to $(q,\beta,\beta)$ costs only a
fixed power of $L$. This proves \eqref{first:legal}; the inverse
image of the neighborhood \eqref{first:legal} under
$q\nabla\kappa$ lies in
$\Omega_{D,P,H}$.

For the shifts used in the proof,
\begin{equation}\label{app:shift-neighbourhood}
 \frac{|\ell_1|}{q}+\frac{|\ell_2|+|\ell_3|}{\beta}
 \ll N^{-9/25}/L.
\end{equation}
Both $N^{-9/25}/L$ in \eqref{app:shift-neighbourhood} and
$\rho=N^{-1/4}$ from \S\ref{sec:partition} are $o(L^{-B_*})$,
where $L^{-B_*}$ is the radius in \eqref{first:legal}.
Choose $x_0\in\supp W\cap Q_j$ for a retained index $j$.
For every $x\in4Q_j$ and $0\le t\le1$,
\[
 |S^{-1}(x+t\ell-x_0)|
 \le C\rho+|S^{-1}\ell|<L^{-B_*}
\]
for sufficiently large $N$. Thus $4Q_j$, $4Q_j+\ell$, and every
intervening shift segment lie in the $(a,b,s)$ neighborhood
\eqref{first:legal}.

\subsection{The transformed amplitude}\label{app:amplitude}

Let $\omega\in C_c^\infty(\Omega_{D,P,H})$ satisfy
$D^iP^jH^k|\partial_d^i\partial_n^j\partial_h^k\omega|\ll L^B$ for every fixed order.
For $(a,b,s)\in q\nabla\kappa(\Omega_{D,P,H})$, let
$(d,n,h)\in\Omega_{D,P,H}$ be its inverse image under
$q\nabla\kappa$, and define
\[
 W(a,b,s)=\frac{\sqrt D}{N^2}
 \frac{\omega(d,n,h)}{\sqrt{|\det\nabla^2\kappa(d,n,h)|}}.
\]
Since $\det\nabla^2\kappa=D\det K_\kappa/(P^2H^2)$, the preceding derivative bounds imply
\begin{equation}\label{app:amplitude-coarse}
 q^i\beta^{j+k}|\partial_a^i\partial_b^j\partial_s^kW|
 \ll L^B\frac{PH}{N^2}\ll L^B.
\end{equation}
The extension of $W$ by zero outside $q\nabla\kappa(\Omega_{D,P,H})$
is smooth.
The product rule gives, for $W_\ell(x)=W(x+\ell)\overline{W(x)}$,
\begin{equation}\label{app:shift-amplitude-coarse}
 q^i\beta^{j+k}|\partial_a^i\partial_b^j\partial_s^kW_\ell|
 \ll L^B.
\end{equation}

\section{Derivatives of the dual phase}\label{mat:appendix}

Fix $x_0=(a_0,b_0,s_0)\in q\nabla\kappa(\Omega_{D,P,H})$ and let
$\partial_i=V_{i,0}\cdot\nabla_{a,b,s}$, with $V_{i,0}$ as in
\eqref{main:frame}.
Quantities with subscript $0$ remain fixed. For
$x=(a,b,s)\in q\nabla\kappa(\Omega_{D,P,H})$, the unsubscripted
root variables are evaluated at its inverse image under
$q\nabla\kappa$.

\subsection{Matrix formulas before evaluation at the base point}

Substitute \eqref{app:four-exponentials} into the root equation and differentiate.
The coefficient of $\dd z$ is $\Delta/z$.
The coefficients of $\dd a$, $\dd b$ and $\dd s$ are, respectively,
$-zd/q$, $-zn(1+\log n)/q$ and
$-z\{v(1+\log v)-n(1+\log n)\}/q$.
Thus
\begin{equation}\label{mat:gradz}
 \nabla z=\frac{z^2}{q\Delta}
 \begin{pmatrix}d\\n(1+\log n)\\v(1+\log v)-n(1+\log n)\end{pmatrix},
 \qquad \nabla=\nabla_{a,b,s}.
\end{equation}
Differentiating $d=\kappa+v-n-w$, $n$ and $h=v-n$ gives
\begin{equation}\label{mat:dualfactor}
 \nabla^2\ps=-\frac zq
 \begin{pmatrix}d&n&h\\n&-n&n\\h&n&h\end{pmatrix}
 -\frac{q\Delta}{z^3}(\nabla z)(\nabla z)^T.
\end{equation}
Hence, for $x\in q\nabla\kappa(\Omega_{D,P,H})$, the entries
in \eqref{main:Adef} satisfy
\begin{align}
 A_{11}={}&\left(\frac z{z_0}\right)^2\left(1+\frac d\Delta\right),\label{mat:A11}\\
 A_{12}={}&\left(\frac z{z_0}\right)^2\sqrt{\frac{Lh_0}{v_0z_0}}
 \left[-1+\frac{d_0n}{n_0d}
 +\frac{-d+(d_0/n_0)n(1+\log n)}\Delta\right],\label{mat:A12}\\
 A_{13}={}&\left(\frac z{z_0}\right)^2\sqrt{\frac{Lh_0}{v_0z_0}}
 \left[-1+\frac{d_0h}{h_0d}
 +\frac{-d+(d_0/h_0)\{v(1+\log v)-n(1+\log n)\}}\Delta\right],\label{mat:A13}\\
 A_{22}={}&\left(\frac z{z_0}\right)^2\frac{Lh_0}{v_0z_0}
 \left[1-\left(\frac{2d_0}{n_0}+\frac{d_0^2}{n_0^2}\right)\frac nd
 +\frac{[-d+(d_0/n_0)n(1+\log n)]^2}{d\Delta}\right],\label{mat:A22}\\
 A_{23}={}&\left(\frac z{z_0}\right)^2\frac{Lh_0}{v_0z_0}
 \left[1-\frac{d_0n}{n_0d}-\frac{d_0h}{h_0d}+\frac{d_0^2n}{n_0h_0d}\right.\notag\\*
 &\left.\hspace{8mm}+
 \frac{[-d+(d_0/n_0)n(1+\log n)]
 [-d+(d_0/h_0)\{v(1+\log v)-n(1+\log n)\}]}{d\Delta}\right],\label{mat:A23}\\
 A_{33}={}&\left(\frac z{z_0}\right)^2\frac{Lh_0}{v_0z_0}
 \left[1-\left(\frac{2d_0}{h_0}-\frac{d_0^2}{h_0^2}\right)\frac hd\right.\notag\\*
 &\left.\hspace{8mm}+
 \frac{[-d+(d_0/h_0)\{v(1+\log v)-n(1+\log n)\}]^2}{d\Delta}\right].\label{mat:A33}
\end{align}
All derivatives below are taken before evaluation at $x_0$.

\subsection{Values and first derivatives at the base point}

Evaluation at $x_0$ gives
\begin{equation}\label{mat:baseentries}
 \begin{aligned}
 A_{11,0}&=1+\frac{d_0}{\Delta_0},\\
 A_{12,0}&=\sqrt{\frac{Lh_0}{v_0z_0}}
                  \frac{d_0\log n_0}{\Delta_0},\\
 A_{13,0}&=\sqrt{\frac{Lh_0}{v_0z_0}}
                  \frac{d_0(v_0\log v_0-n_0\log n_0)}{h_0\Delta_0},\\
 A_{22,0}&=\frac{Lh_0}{v_0z_0}
 \left[-1-\frac{d_0}{n_0}+\frac{d_0(\log n_0)^2}{\Delta_0}\right].
 \end{aligned}
\end{equation}
By \eqref{app:Delta-refined} and \eqref{app:elementary-scales}, the bracketed expressions in
\eqref{mat:A22}--\eqref{mat:A33}, evaluated at $x_0$, are $O((1+z_0)/L)$.
Consequently,
\begin{equation}\label{mat:Abounds}
 1\le A_{11,0}\le C,\qquad |A_{ij,0}|\le C,\qquad
 |A_{12,0}|+|A_{13,0}|=o(1).
\end{equation}
The $o(1)$ terms are uniform for
$(d_0,n_0,h_0)\in\Omega_{D,P,H}$.

To bound $A_{22,0}$ away from zero, use
\[
 \frac\Delta d\ge(1+\log n)(1+\log v),\qquad
 z\le\log(v/n),\qquad
 \frac hv[1+\log(v/n)]\ge\log(v/n).
\]
It follows that
\begin{equation}\label{mat:A22negative}
 -A_{22,0}\ge\frac{Lh_0}{v_0z_0}
 \frac{1+\log(v_0/n_0)}{1+\log v_0}
 \ge\frac L{1+\log v_0}\ge\tfrac12
\end{equation}
for sufficiently large $N$.
Also, \eqref{app:dual-det} gives
\[
 \det V_0=\frac{q^3Ld_0^2}{n_0v_0z_0^4},\qquad
 \det\mathsf A(x_0)=-\frac{d_0L^2\kappa_0w_0}{n_0v_0\Delta_0},\qquad
 -\det\mathsf A(x_0)\asymp1.
\]
Expansion along the first row, together with \eqref{mat:Abounds}, yields
\begin{equation}\label{mat:blocknegative}
 A_{22,0}A_{33,0}-A_{23,0}^2\le-c.
\end{equation}

It remains to prove
\begin{equation}\label{mat:derivativesmall}
 (\partial_1A_{ij})(x_0)=o(1)\quad(1\le i,j\le3),\qquad
 (\partial_kA_{ij})(x_0)=o(1)\quad(i,j,k\in\{2,3\}).
\end{equation}
Differentiating $\nabla\ps=-(d,n,h)^T$ along the fixed directions gives
\begin{equation}\label{mat:rates}
 \begin{aligned}
 \left.\frac{\partial_kd}{d}\right|_0&=A_{1k,0},\\
 \left.\frac{\partial_kn}{n}\right|_0
 &=A_{1k,0}+\sqrt{\frac{v_0z_0}{Lh_0}}A_{2k,0},\\
 \left.\frac{\partial_kh}{h}\right|_0
 &=A_{1k,0}+\sqrt{\frac{v_0z_0}{Lh_0}}A_{3k,0}.
 \end{aligned}
\end{equation}
These ratios are $O(1)$, and are $o(1)$ for $k=2,3$, since
$\sqrt{v_0z_0/(Lh_0)}\ll\sqrt{(1+z_0)/L}=o(1)$.
The chain rule and \eqref{app:Delta-refined} give
\begin{equation}\label{mat:Deltadirection}
 \frac{\partial_k\Delta}{\Delta}=\frac{\partial_kd}{d}
 +O\left(\frac1L\left[
 \frac{|\partial_kd|}{d}+\frac{|\partial_kn|}{n}
 +\frac{|\partial_kh|}{h}\right]\right).
\end{equation}

By \eqref{mat:gradz}, the fixed directional derivatives of $z$
satisfy, for $x\in q\nabla\kappa(\Omega_{D,P,H})$,
\begin{equation}\label{mat:zdirections}
 \begin{aligned}
 \partial_1z&=\frac{z^2d}{z_0\Delta},\\
 \partial_2z&=\frac{z^2}{\Delta}\sqrt{\frac{Lh_0}{v_0z_0^3}}
 \left[-d+\frac{d_0}{n_0}n(1+\log n)\right],\\
 \partial_3z&=\frac{z^2}{\Delta}\sqrt{\frac{Lh_0}{v_0z_0^3}}
 \left[-d+\frac{d_0}{h_0}\{v(1+\log v)-n(1+\log n)\}\right].
 \end{aligned}
\end{equation}
At $x_0$, the ratios $\partial_i z/z$ are $d_0/\Delta_0$, $A_{12,0}$ and $A_{13,0}$, all $o(1)$.
For the last two brackets, the ratios of their $\partial_k$ derivatives to their values, all evaluated at $x_0$, are
\begin{equation}\label{mat:bracketrates}
 \begin{aligned}
 &\left.\left[\frac{\partial_kn}{n}
 +\frac{2\partial_kn/n-\partial_kd/d}{\log n}\right]\right|_0,\\
 &\left.\left[\frac{\partial_kh}{h}
 +\frac{2\partial_kh/h-\partial_kd/d
 +(n/h)\log(v/n)(\partial_kn/n-\partial_kh/h)}{(v\log v-n\log n)/h}\right]\right|_0.
 \end{aligned}
\end{equation}
Both are $O(1)$, since $(n/h)\log(v/n)\le1$ and the denominators are $\asymp L$.
Differentiating \eqref{mat:zdirections} and using \eqref{mat:Deltadirection} now gives
$(\partial_i\partial_jz)(x_0)/z_0=o(1)$ for every $i,j$.
Finally, the identity $\partial_1\ps=(z/z_0)\ps$, valid for
$x\in q\nabla\kappa(\Omega_{D,P,H})$, implies
\[
 (\partial_1A_{ij})(x_0)
 =\left.\left[
 \frac{\partial_i z}{z}\frac{\partial_j\ps}{\ps}
 +\frac{\partial_j z}{z}\frac{\partial_i\ps}{\ps}
 +\frac{\partial_i\partial_jz}{z}\right]\right|_0=o(1).
\]

For $k=2,3$, differentiate the square term in \eqref{mat:A22}
with the subscript-$0$ coefficients fixed.
Equations \eqref{mat:bracketrates} and \eqref{mat:Deltadirection} give
\begin{align*}
 \left.\partial_k
 \frac{[-d+(d_0/n_0)n(1+\log n)]^2}{d\Delta}\right|_0
 ={}&\frac{2d_0(\log n_0)^2}{\Delta_0}
 \left.\left(\frac{\partial_kn}{n}-\frac{\partial_kd}{d}\right)\right|_0\\*
 &+O\left(\frac1L
 \left.\left[\frac{|\partial_kd|}{d}+\frac{|\partial_kn|}{n}
 +\frac{|\partial_kh|}{h}\right]\right|_0\right).
\end{align*}
Combining this with the derivative of the linear term gives the coefficient
\[
 -\left(2+\frac{d_0}{n_0}\right)
 +\frac{2d_0(\log n_0)^2}{\Delta_0}
 =\frac{2v_0z_0}{Lh_0}A_{22,0}+\frac{d_0}{n_0}.
\]
After multiplication by $Lh_0/(v_0z_0)$, this is
$2A_{22,0}+Lh_0d_0/(v_0z_0n_0)=O(1)$ by \eqref{app:elementary-scales}.
Differentiating the remaining factor $(z/z_0)^2$ gives the first identity below.
The same calculation for \eqref{mat:A33}, using the second ratio in
\eqref{mat:bracketrates}, gives the second:
\begin{align}
 (\partial_kA_{22})(x_0)
 ={}&2A_{22,0}\left.\frac{\partial_kz}{z}\right|_0
 +\left(2A_{22,0}+\frac{Lh_0d_0}{v_0z_0n_0}\right)
 \left.\left(\frac{\partial_kn}{n}-\frac{\partial_kd}{d}\right)\right|_0\notag\\
 &+O\left(\frac{h_0}{v_0z_0}
 \left.\left[\frac{|\partial_kd|}{d}+\frac{|\partial_kn|}{n}
 +\frac{|\partial_kh|}{h}\right]\right|_0\right),\label{mat:A22derivative}\\
 (\partial_kA_{33})(x_0)
 ={}&2A_{33,0}\left.\frac{\partial_kz}{z}\right|_0
 +\left(2A_{33,0}-\frac{Ld_0}{v_0z_0}\right)
 \left.\left(\frac{\partial_kh}{h}-\frac{\partial_kd}{d}\right)\right|_0\notag\\
 &+O\left(\frac{h_0}{v_0z_0}
 \left.\left[\frac{|\partial_kd|}{d}+\frac{|\partial_kn|}{n}
 +\frac{|\partial_kh|}{h}\right]\right|_0\right).\label{mat:A33derivative}
\end{align}
The coefficients are $O(1)$ by \eqref{app:elementary-scales}, so \eqref{mat:rates} makes both right sides $o(1)$.
The fixed directional derivatives commute, and
$(\partial_2\ps)(x_0)=(\partial_3\ps)(x_0)=0$.
Hence
\[
 (\partial_2A_{23})(x_0)=(\partial_3A_{22})(x_0),\qquad
 (\partial_3A_{23})(x_0)=(\partial_2A_{33})(x_0).
\]
This proves \eqref{mat:derivativesmall}, and hence \eqref{main:Abounds}.

\subsection{A fourth derivative estimate}\label{fourth:appendix}

Uniformly for $(a,b,s)\in q\nabla\kappa(\Omega_{D,P,H})$,
\begin{equation}\label{fourth:pure}
 q^4\partial_a^4\ps=-qd z^3\bigl(1+O(L^{-2})\bigr).
\end{equation}
To prove this, put
\[
 r=\frac d\Delta,\qquad \mathcal D=\frac qz\partial_a,\qquad
 \Xi_m=n(1+\log n)^m+w(1+\log w)^m
       -\kappa(1+\log\kappa)^m-v(1+\log v)^m.
\]
Unlike $\partial_1$, the operator $\mathcal D$ has variable coefficients.
The root equation gives
$\Xi_0=\Xi_1=-d$ and $\Xi_2=\Delta$.
For $m=3,4$, differentiation at fixed $n,h$ gives
\begin{align*}
 \partial_d\Xi_m
 ={}&\frac{(1+\log\kappa)(1+\log w)}z
 \bigl[(1+\log w)^{m-1}-(1+\log\kappa)^{m-1}\\
 &\hspace{22mm}+m(1+\log w)^{m-2}-m(1+\log\kappa)^{m-2}\bigr]
 \ll L^m.
\end{align*}
Each difference of powers contains the factor $z$.
Since $\Xi_m=0$ at $d=0$, integration gives
\begin{equation}\label{fourth:moments}
 |\Xi_3|\ll dL^3,\qquad |\Xi_4|\ll dL^4,
 \qquad r\asymp L^{-2}.
\end{equation}
Equations \eqref{app:four-exponentials} and \eqref{mat:gradz} imply
\[
 \mathcal Dz=zr,\qquad
 \mathcal Dx=x\bigl(1+r(1+\log x)\bigr)
 \quad(x=n,w,\kappa,v).
\]
Thus, for $m=2,3$,
\[
 \mathcal D\Xi_m
 =\Xi_m+m\Xi_{m-1}+r(\Xi_{m+1}+m\Xi_m).
\]
In particular, $\mathcal D\Delta=\Delta+r\Xi_3$.
Also, $\mathcal D\ps=\ps$ and $d=-z\ps/q$, so $\mathcal Dd=d(1+r)$.
It follows that
\begin{align}
 \mathcal Dr
 &=r^2\left(1-\frac{\Xi_3}{\Delta}\right)=O(L^{-3}),\label{fourth:rho1}\\
 \mathcal D\left(\frac{\Xi_3}{\Delta}\right)
 &=3+r\left[\frac{\Xi_4}{\Delta}+3\frac{\Xi_3}{\Delta}
       -\left(\frac{\Xi_3}{\Delta}\right)^2\right]=O(1).
 \label{fourth:ratio}
\end{align}
Differentiating the first identity once more gives $\mathcal D^2r=O(L^{-4})$.
Since $q\partial_a=z\mathcal D$, the product rule gives
\begin{align*}
 q^2\partial_a^2\ps&=z^2\ps(1+r),\\
 q^3\partial_a^3\ps&=z^3\ps(1+3r+2r^2+\mathcal Dr),\\
 q^4\partial_a^4\ps
 &=z^4\ps\bigl(1+6r+11r^2+6r^3+(4+7r)\mathcal Dr+\mathcal D^2r\bigr).
\end{align*}
Together with $\ps=-qd/z$, these identities prove \eqref{fourth:pure}.

For a shift satisfying \eqref{main:shifts} and \eqref{main:dominant},
and for $x=(a,b,s)$ satisfying
\[
 \inf_{y\in\supp W}|S^{-1}(x-y)|<\tfrac12L^{-B_*},
\]
we obtain
\begin{equation}\label{fourth:shiftthird}
 q^3|\partial_a^3\ph|\asymp_L F_\ell,
 \qquad \sgn(\partial_a^3\ph)=-\sgn(\ell_1).
\end{equation}
By \eqref{app:shift-neighbourhood}, the shift segment $x+t\ell$,
$0\le t\le1$, lies in the $(a,b,s)$ neighborhood
\eqref{first:legal}, and its inverse image under $q\nabla\kappa$
lies in $\Omega_{D,P,H}$. Therefore
\begin{align*}
 q^3\partial_a^3\ph(x)
 &=-\ell_1\int_0^1d(x+t\ell)z(x+t\ell)^3
                    \bigl(1+O(L^{-2})\bigr)\,\dd t\\
 &\quad+O\left(L^B\frac NL(|\ell_2|+|\ell_3|)\right).
\end{align*}
The error follows from \eqref{app:dual-coarse} and is at most $CL^BN^{11/10}/L$.
The integral is positive and $\asymp_L D$.
For sufficiently large fixed $K_*$, the error is smaller than half the main term.
Since $F_\ell\asymp D|\ell_1|$ under \eqref{main:shifts} and
\eqref{main:dominant}, \eqref{fourth:shiftthird} follows.

\section{A multidimensional \texorpdfstring{$B$}{B}-process}
\label{app:uniform-stationary-phase}

Proposition~\ref{prop:multidimensional-B} below gives the common
$B$-process used in the text. We first prove the required uniform
integral estimates and frequency-tail bound. For the standard
stationary-phase expansion, see
H\"ormander~\cite[Section~7.7]{Hormander}.

\begin{lemma}[Nonstationary phase]\label{lem:nonstationary-phase}
Let $m\in\{2,3\}$ and let $\varphi$ be real and smooth on an open
set $\Omega\subset\mathbb R^m$. Let $k\ge1$ be a fixed integer, and suppose that
for some $0<p\ll_L1$,
\[
 \|(\nabla^2\varphi)^{-1}\|\ll_Lp^{-1},
 \qquad \|\partial^r\varphi\|\ll_L1\quad(2\le r\le k+1)
\]
on $\Omega$. If $\zeta\in C_c^\infty(\Omega)$ satisfies
$\operatorname{supp}\zeta\subset[-R_0,R_0]^m$ for a fixed $R_0>0$,
and $\nabla\varphi\ne0$ on $\operatorname{supp}\zeta$, then
\begin{equation}\label{eq:nonstationary-phase}
 \left|\int\zeta(u)\ee(T\varphi(u))\,du\right|
 \ll_{L,k}(Tp^2)^{-k}\|\zeta\|_{C^{2k}(\Omega)}\qquad(T\ge1).
\end{equation}
No uniform lower bound for $|\nabla\varphi|$ on
$\operatorname{supp}\zeta$ is required. In particular, $p=1$
gives the bound $O_{L,k}(T^{-k}\|\zeta\|_{C^{2k}})$.
\end{lemma}

\begin{proof}
Put $Q=|\nabla\varphi|^2$. For an integer $j\ge1$ and a real smooth
function $F:\Omega\to\mathbb R$ with
$\operatorname{supp}F\subset\operatorname{supp}\zeta$, the maximum
of $|F|^2/Q^j$ over $\operatorname{supp}\zeta$ is attained.
All suprema of quotients involving $Q$ below are taken over
$\operatorname{supp}\zeta$, where $Q>0$.
At a positive maximum, differentiation gives
\[
 \nabla F=jF\frac{\nabla^2\varphi\,\nabla\varphi}{Q}.
\]
Multiplying by $(\nabla^2\varphi)^{-1}$ therefore yields
\[
 \sup\frac{|F|}{Q^{j/2}}
 \ll_L\frac1{jp}\sup\frac{|\nabla F|}{Q^{(j-1)/2}}.
\]
Applying this successively to the components of the derivatives gives
\begin{equation}\label{eq:nonstationary-quotient}
 \sup\frac{|F|}{Q^{j/2}}\ll_{L,j}p^{-j}\|F\|_{C^j(\Omega)}
 \qquad(j\ge0).
\end{equation}
The same bound holds for complex $F$ by taking real and imaginary parts.
The minimum of $Q$ on $\operatorname{supp}\zeta$ is positive,
but does not enter this bound.

On the open set $\{u\in\Omega:Q(u)>0\}$, which contains
$\operatorname{supp}\zeta$, set
$\mathcal D F=\operatorname{div}(F\nabla\varphi/Q)$.
Each $\mathcal D^j\zeta$ has support contained in
$\operatorname{supp}\zeta$ and is extended by zero to $\Omega$.
Integration by parts $k$ times gives
\[
 \int\zeta\,\ee(T\varphi)\,du
 =(-2\pi iT)^{-k}\int\mathcal D^k\zeta\,\ee(T\varphi)\,du.
\]
Each term of $\mathcal D^k\zeta$ is a product of derivatives of
$\varphi$ of orders between two and $k+1$, a derivative
$\partial^\alpha\zeta$, and a smooth homogeneous function of
$\nabla\varphi$ of degree $-j$, where $|\alpha|+j\le2k$.
Equation~\eqref{eq:nonstationary-quotient}, applied to
$\partial^\alpha\zeta$, bounds this term by
$O_{L,k}(p^{-j}\|\zeta\|_{C^{2k}(\Omega)})$.
Since $p\ll_L1$ and $0\le j\le2k$, this is
$O_{L,k}(p^{-2k}\|\zeta\|_{C^{2k}(\Omega)})$.
Integrating over the fixed box $[-R_0,R_0]^m$ proves
\eqref{eq:nonstationary-phase}.
\end{proof}

\begin{lemma}\label{lem:uniform-stationary-phase}
Let $m\in\{2,3\}$, let $\Omega\subset\mathbb R^m$ be open, and let
$K\subset\Omega$ be compact with $K\subset[-R_0,R_0]^m$ for a fixed $R_0>0$.
Suppose that $\operatorname{dist}(K,\Omega^c)\ge c>0$, where $c$ is
fixed. Let $\varphi$ be real and smooth on $\Omega$, with at most one
stationary point in $\Omega$, and assume throughout $\Omega$ that
\[
 \|(\nabla^2\varphi)^{-1}\|\ll_L1,
 \qquad \|\partial^r\varphi\|\ll_L1\quad(2\le r\le12).
\]
For $\zeta\in C_c^\infty(\Omega)$ with
$\operatorname{supp}\zeta\subset K$ and $\|\zeta\|_{C^{12}(\Omega)}\ll_L1$,
we have, uniformly for $T\ge1$,
\begin{equation}\label{eq:uniform-SP}
 \int \zeta(u)\ee(T\varphi(u))\,du
 =\frac{\zeta(u_*)\ee(T\varphi(u_*)+\sigma/8)}
        {T^{m/2}\sqrt{|\det\nabla^2\varphi(u_*)|}}
        +O_L(T^{-m/2-1}).
\end{equation}
Here $u_*$ is the stationary point of $\varphi$ in $\Omega$,
$\sigma$ is the signature of $\nabla^2\varphi(u_*)$, and the main
term is zero if no stationary point exists in $\Omega$ or if
$u_*\notin K$.
\end{lemma}

\begin{proof}
If there is no stationary point in $K$,
Lemma~\ref{lem:nonstationary-phase} with $p=1$ and $k=4$ proves the assertion.
Otherwise put $H_*=\nabla^2\varphi(u_*)$ and $Q_m=[-1,1]^m$.
Choose a fixed $\chi_{\mathrm{sp}}\in C_c^\infty(\operatorname{int}Q_m)$
equal to one on $\tfrac12Q_m$. Take $r=c_1L^{-C}$, with $c_1>0$
small and $C$ fixed and sufficiently large.
Then $u_*+4rQ_m\subset\Omega$.
Split the amplitude $\zeta$ as
\[
 \zeta(u)\chi_{\mathrm{sp}}((u-u_*)/r)
 +\zeta(u)\bigl(1-\chi_{\mathrm{sp}}((u-u_*)/r)\bigr).
\]
The function
$u\mapsto\zeta(u)(1-\chi_{\mathrm{sp}}((u-u_*)/r))$ has support
contained in
$\operatorname{supp}\zeta\setminus\operatorname{int}(u_*+\tfrac r2Q_m)$
and has $C^8(\Omega)$ norm $O_L(1)$. The integral with this amplitude is
$O_L(T^{-4})$ by Lemma~\ref{lem:nonstationary-phase} with $p=1$.

For the integral with amplitude
$u\mapsto\zeta(u)\chi_{\mathrm{sp}}((u-u_*)/r)$, set
\[
 w=(u-u_*)/r,\qquad
 \psi(w)=\frac{\varphi(u_*+rw)-\varphi(u_*)}{r^2},\qquad
 H(w)=2\int_0^1(1-t)\nabla^2\psi(tw)\,dt.
\]
Taylor's formula gives $\psi(w)=w^TH(w)w/2$ and $H(0)=H_*$.
The derivative bounds and the choice of $r$ give
\[
 \|H_*^{-1}(H-H_*)\|_{C^9(4Q_m)}=o(1).
\]
The matrix power series
$P(w)=(I+H_*^{-1}(H(w)-H_*))^{1/2}$ consequently converges and
satisfies $P(w)^TH_*P(w)=H(w)$. Indeed, symmetry of $H$ and $H_*$
gives $P^TH_*=H_*P$ by the same power series, and hence
$P^TH_*P=H_*P^2=H$. This also applies when $H_*$ is indefinite.

The map $w\mapsto P(w)w$ has derivative $I+o(1)$ on the convex
box $4Q_m$. It is injective on $4Q_m$, and its inverse on the image
of $\operatorname{int}(4Q_m)$ has bounded derivatives through order
nine. The change of variables transforms the amplitude
$w\mapsto\zeta(u_*+rw)\chi_{\mathrm{sp}}(w)$, including the Jacobian,
into an amplitude $b$. The support of $b$ is contained in
$\{P(w)w:w\in\operatorname{supp}\chi_{\mathrm{sp}}\}$, a compact
subset of the image of $\operatorname{int}(4Q_m)$ under
$w\mapsto P(w)w$, contained in $[-C,C]^m$ for a fixed $C>0$.
Thus $b$ extends by zero to a function in $C_c^8(\mathbb R^m)$ with
$\|b\|_{C^8(\mathbb R^m)}\ll_L1$ and $b(0)=\zeta(u_*)$.
With $\widehat b(\xi)=\int b(v)\ee(-\xi\cdot v)\,dv$, the Gaussian
identity gives, for every $t>0$,
\begin{equation}\label{eq:stationary-gaussian}
 \int b(v)\ee(tv^TH_*v/2)\,dv
 =\frac{\ee(\sigma/8)}{t^{m/2}\sqrt{|\det H_*|}}
   \int\widehat b(\xi)
       \ee\left(-\frac{\xi^TH_*^{-1}\xi}{2t}\right)d\xi.
\end{equation}
It follows first with Gaussian damping and then by passage to the
limit. Since
$\int(1+|\xi|^2)|\widehat b(\xi)|\,d\xi\ll_L1$, Fourier inversion
and $|\ee(s)-1|\le2\pi|s|$ give the main term in
\eqref{eq:uniform-SP} after taking $t=Tr^2$ and multiplying by
$r^m\ee(T\varphi(u_*))$. The error is
$O_L(r^{-2}T^{-m/2-1})=O_L(T^{-m/2-1})$.
\end{proof}

\begin{corollary}\label{cor:absolute-stationary-phase}
Let $m\in\{2,3\}$, $Q_m=[-1,1]^m$, and let $\varphi$ be real and
smooth on an open set containing $4Q_m$.
Suppose that $u_*\in4Q_m$ satisfies $\nabla\varphi(u_*)=0$
and that $B=\nabla^2\varphi(u_*)$ is nonsingular.
There is an absolute constant $c_m>0$ such that, if
\[
 \bigl\|B^{-1}(\nabla^2\varphi-B)\bigr\|_{C^9(4Q_m)}\le c_m,
\]
then every $\zeta\in C_c^\infty(\operatorname{int}Q_m)$ satisfies
\begin{equation}\label{eq:absolute-SP}
 \left|\int\zeta(u)\ee(\varphi(u))\,du\right|
 \ll \frac{\|\zeta\|_{C^8(Q_m)}}{\sqrt{|\det B|}}.
\end{equation}
The implied constant depends only on $m$.
\end{corollary}

\begin{proof}
Since $4Q_m$ is convex, $u_*+t(u-u_*)\in4Q_m$ for
$u\in4Q_m$ and $0\le t\le1$. On $4Q_m$ put
\[
 H(u)=2\int_0^1(1-t)
   \nabla^2\varphi(u_*+t(u-u_*))\,dt.
\]
Taylor's formula gives
$\varphi(u)-\varphi(u_*)=(u-u_*)^TH(u)(u-u_*)/2$.
The relative $C^9$ bound implies
$\|B^{-1}(H-B)\|_{C^9(4Q_m)}\ll_m c_m$.
For sufficiently small $c_m$, the same matrix power series gives
$P(u)=(I+B^{-1}(H(u)-B))^{1/2}$ with $P(u)^TBP(u)=H(u)$.
Since $|u-u_*|\le8\sqrt m$ on $4Q_m$, the map
$u\mapsto P(u)(u-u_*)$ has derivative $I+O_m(c_m)$.
For sufficiently small $c_m$, it is injective on $4Q_m$, and its
inverse on the image of $\operatorname{int}(4Q_m)$ has bounded
derivatives through order nine. Under this change of variables,
the amplitude $\zeta$, including the Jacobian, gives an amplitude
$b$ whose support is contained in
$\{P(u)(u-u_*):u\in\operatorname{supp}\zeta\}$.
This set is a compact subset of the image of
$\operatorname{int}(4Q_m)$ under $u\mapsto P(u)(u-u_*)$ and lies in $[-C,C]^m$
for a fixed $C>0$.
Hence $b$ extends by zero to $\mathbb R^m$ with
$\|b\|_{C^8(\mathbb R^m)}\ll_m\|\zeta\|_{C^8(Q_m)}$, so
$\|\widehat b\|_{L^1(\mathbb R^m)}\ll_m\|\zeta\|_{C^8(Q_m)}$.
Apply \eqref{eq:stationary-gaussian} with $t=1$ and $H_*=B$, and
take absolute values. This introduces no additional dependence on
$\|B^{-1}\|$.
\end{proof}

\begin{lemma}[Summation outside a frequency box]
\label{lem:frequency-tail}
Let $m\in\{2,3\}$, let $K>m$ be an integer, and let
$U\subset\mathbb R^m$ be a fixed bounded open box.
Let $\zeta\in C_c^\infty(U)$ and let $\varphi:U\to\mathbb R$
be smooth. Suppose that $T\ge1$ and
\[
 \|\partial^\alpha\zeta\|_{L^\infty(U)}\ll_L1
       \quad(|\alpha|\le K),\qquad
 \|\partial^\alpha\varphi\|_{L^\infty(U)}\ll_LT
       \quad(2\le|\alpha|\le K+1).
\]
Let $X_1,\ldots,X_m>0$, put
$\mathsf X=\diag(X_1,\ldots,X_m)$, and let $\xi\in\mathbb R^m$.
Assume that, for some $A=CL^B\ge1$ with fixed $C,B$,
\[
 |\nabla\varphi(u)-\mathsf X\xi|_\infty\le AT
 \qquad(u\in\supp\zeta).
\]
For $\nu\in\mathbb Z^m$, define
\[
 I_\nu=(\det\mathsf X)\int_U
 \zeta(u)\ee\bigl(\varphi(u)-(\mathsf X\nu)\cdot u\bigr)\,du,
\]
and, for $j\ge0$, define the frequency boxes
\[
 \mathcal B_j=\left\{\nu\in\mathbb R^m:
 |X_i(\nu_i-\xi_i)|\le2^{j+1}AT\quad(1\le i\le m)\right\}.
\]
Then
\begin{equation}\label{eq:frequency-tail-mode}
 |I_\nu|\ll_{L,K}
 (\det\mathsf X)T^{-K}2^{-Kj}
 \qquad
 \bigl(\nu\in(\mathcal B_{j+1}\setminus\mathcal B_j)
                    \cap\mathbb Z^m\bigr),
\end{equation}
and
\begin{equation}\label{eq:frequency-tail-sum}
 \sum_{\nu\in\mathbb Z^m\setminus\mathcal B_0}|I_\nu|
 \ll_{L,K}(\det\mathsf X)T^{-K}
            \prod_{i=1}^m\left(1+\frac{T}{X_i}\right).
\end{equation}
The bounds are uniform in $T$, $X_1,\ldots,X_m$, and $\xi$;
the implied constants depend only on $m,K,U$ and the constants
in the hypotheses.
\end{lemma}

\begin{proof}
Fix $\nu\in(\mathcal B_{j+1}\setminus\mathcal B_j)\cap\mathbb Z^m$.
For at least one coordinate $i$,
\[
 |X_i(\nu_i-\xi_i)|>2^{j+1}AT.
\]
Consequently, throughout $\supp\zeta$,
\begin{align*}
 |\partial_i\varphi(u)-X_i\nu_i|
 &\ge |X_i(\nu_i-\xi_i)|
           -|\partial_i\varphi(u)-X_i\xi_i|\\
 &\ge(2^{j+1}-1)AT\ge2^jAT.
\end{align*}
This coordinate $i$ is fixed for the entire integral.

Put $g_\nu(u)=\partial_i\varphi(u)-X_i\nu_i$.
For $1\le r\le K$, the derivative hypotheses give
$|\partial_i^r g_\nu|\ll_LT$.
Differentiating $g_\nu^{-1}$ therefore gives
\[
 |\partial_i^r(g_\nu^{-1})|
 \ll_{L,K}(2^jT)^{-1}\qquad(0\le r\le K)
\]
on $\supp\zeta$.
On the open set $\{u\in U:g_\nu(u)\ne0\}$, which contains
$\supp\zeta$, set
\[
 \mathcal D_\nu F=\partial_i(F/g_\nu).
\]
Each $\mathcal D_\nu^r\zeta$ is supported in $\supp\zeta$ and
is extended by zero to $U$.
Since
\[
 \partial_i\ee\bigl(\varphi(u)-(\mathsf X\nu)\cdot u\bigr)
 =2\pi i g_\nu(u)
       \ee\bigl(\varphi(u)-(\mathsf X\nu)\cdot u\bigr),
\]
integration by parts $K$ times yields
\[
 I_\nu=(\det\mathsf X)(-2\pi i)^{-K}
 \int_U\mathcal D_\nu^K\zeta(u)
       \ee\bigl(\varphi(u)-(\mathsf X\nu)\cdot u\bigr)\,du.
\]
There are no boundary terms because $\zeta$ is compactly supported
in $U$. The product rule and the preceding bounds imply
\[
 \|\mathcal D_\nu^K\zeta\|_{L^\infty(U)}
 \ll_{L,K}(2^jT)^{-K},
\]
which proves \eqref{eq:frequency-tail-mode}.

The $i$th side length of $\mathcal B_{j+1}$ is
$2^{j+3}AT/X_i$. Hence
\begin{align*}
 \#\bigl((\mathcal B_{j+1}\setminus\mathcal B_j)
                   \cap\mathbb Z^m\bigr)
 &\le\#(\mathcal B_{j+1}\cap\mathbb Z^m)\\
 &\ll\prod_{i=1}^m\left(1+\frac{2^{j+3}AT}{X_i}\right)\\
 &\ll_L2^{mj}\prod_{i=1}^m\left(1+\frac{T}{X_i}\right).
\end{align*}
The disjoint sets $\mathcal B_{j+1}\setminus\mathcal B_j$,
$j\ge0$, cover $\mathbb R^m\setminus\mathcal B_0$.
Combining this count with \eqref{eq:frequency-tail-mode} gives
\[
 \sum_{\nu\in\mathbb Z^m\setminus\mathcal B_0}|I_\nu|
 \ll_{L,K}(\det\mathsf X)T^{-K}
 \prod_{i=1}^m\left(1+\frac{T}{X_i}\right)
 \sum_{j\ge0}2^{-j(K-m)}.
\]
The geometric series converges because $K>m$, proving
\eqref{eq:frequency-tail-sum}.
\end{proof}

\begin{proposition}[A multidimensional $B$-process]
\label{prop:multidimensional-B}
Let $m\in\{2,3\}$, let $U\subset\mathbb R^m$ be a fixed bounded
open box, and let $K\subset U$ be a fixed compact set with
$\operatorname{dist}(K,U^c)>0$.
Let $c\in\mathbb R^m$, $\mathsf X=\diag(X_1,\ldots,X_m)$ with
$X_i>0$, and $T\ge1$.
Suppose that $f$ is real and smooth on $c+\mathsf XU$, and that
$w\in C_c^\infty(c+\mathsf XU)$.
Write
\[
 \varphi(u)=f(c+\mathsf Xu),\qquad
 \zeta(u)=w(c+\mathsf Xu),
\]
and assume that
\[
 \supp\zeta\subset K,\qquad
 \|\zeta\|_{C^{16}(U)}\ll_L1,\qquad
 \|\partial^r\varphi\|_{L^\infty(U)}\ll_LT
 \quad(2\le r\le12).
\]
Suppose also that, for some $\xi\in\mathbb R^m$ and
$A=CL^B\ge1$ with fixed $C,B$,
\[
 |\nabla\varphi(u)-\mathsf X\xi|_\infty\le AT
 \qquad(u\in\supp\zeta),
\]
and put
\[
 \mathcal B=\{k\in\mathbb R^m:
       |\mathsf X(k-\xi)|_\infty\le2AT\}.
\]
The physical-space weighted summand is extended by zero outside
$c+\mathsf XU$.

\smallskip
\noindent\textup{(i) Uniformly nondegenerate Hessian.}
Assume that $\nabla f$ is injective on $c+\mathsf XU$ and that
\[
 \bigl\|(T^{-1}\nabla^2\varphi(u))^{-1}\bigr\|\ll_L1
 \qquad(u\in U).
\]
For $k\in\nabla f(c+\mathsf XU)$, let $x(k)$ be the unique solution
of $\nabla f(x)=k$ in $c+\mathsf XU$, and let $\sigma$ be the
constant signature of $\nabla^2f$ on this box. Then
\begin{equation}\label{eq:general-B-formula}
 \sum_{n\in\mathbb Z^m}w(n)\ee(f(n))
 =\sum_{k\in\mathcal B\cap\mathbb Z^m}
   \frac{w(x(k))\ee\bigl(f(x(k))-k\cdot x(k)+\sigma/8\bigr)}
        {\sqrt{|\det\nabla^2f(x(k))|}}+E,
\end{equation}
where a summand is zero if $k\notin\nabla f(c+\mathsf XU)$, and
\begin{equation}\label{eq:general-B-error}
 |E|\ll_L(\det\mathsf X)T^{-m/2-1}
          \prod_{i=1}^m\left(1+\frac{T}{X_i}\right).
\end{equation}
The weighted dual summand extends smoothly by zero outside the
gradient image and is supported in $\nabla f(\supp w)$.
Thus the dual sum may equivalently be taken over all integer
frequencies, or over any set containing $\nabla f(\supp w)$.
No convexity of the gradient image is required.
In particular,
\begin{equation}\label{eq:general-B-bound}
 \left|\sum_{n\in\mathbb Z^m}w(n)\ee(f(n))\right|
 \ll_L(\det\mathsf X)T^{-m/2}
          \prod_{i=1}^m\left(1+\frac{T}{X_i}\right).
\end{equation}

\smallskip
\noindent\textup{(ii) An absolute bound for a nearly constant Hessian.}
Instead of the additional assumptions in \textup{(i)}, suppose that
$U=(-4,4)^m$, $K=[-1,1]^m$, and $\supp\zeta\subset(-1,1)^m$.
Assume that $f$ is smooth on an open neighborhood of
$c+\mathsf X[-4,4]^m$, and put
\[
 H_0=\mathsf X\nabla^2f(c)\mathsf X=\nabla^2\varphi(0).
\]
Suppose that $H_0$ is nonsingular and, for some $0<p\ll_L1$,
\[
 \|H_0^{-1}\|\ll_L(Tp)^{-1},\qquad
 \bigl\|H_0^{-1}(\nabla^2\varphi-H_0)\bigr\|_{C^9([-4,4]^m)}
       \le \varepsilon_m,
\]
where $\varepsilon_m>0$ is sufficiently small and depends only on $m$.
Then
\begin{equation}\label{eq:general-B-nearconstant}
 \left|\sum_{n\in\mathbb Z^m}w(n)\ee(f(n))\right|
 \ll_L(\det\mathsf X)
       \prod_{i=1}^m\left(1+\frac{T}{X_i}\right)
       \left\{|\det H_0|^{-1/2}+(Tp^2)^{-8}\right\}.
\end{equation}
All implied constants are uniform in $c,\mathsf X,T,\xi$ and,
in \textup{(ii)}, $p$; they depend only on the fixed sets and the
constants in the stated hypotheses.
\end{proposition}

\begin{proof}
For $k\in\mathbb Z^m$, a single change of variables gives
\[
 \begin{aligned}
 I_k&=\int_{\mathbb R^m}w(x)\ee\bigl(f(x)-k\cdot x\bigr)\,\dd x\\
 &=(\det\mathsf X)\ee(-k\cdot c)
   \int_U\zeta(u)
       \ee\bigl(\varphi(u)-(\mathsf Xk)\cdot u\bigr)\,\dd u.
 \end{aligned}
\]
Poisson summation expresses the sum in the proposition as
$\sum_{k\in\mathbb Z^m}I_k$.
The box $\mathcal B$ contains
$O_L(\prod_i(1+T/X_i))$ integer points.
Lemma~\ref{lem:frequency-tail}, with eight integrations by parts,
applies under the common hypotheses and gives
\[
 \sum_{k\in\mathbb Z^m\setminus\mathcal B}|I_k|
 \ll_L(\det\mathsf X)T^{-8}
             \prod_{i=1}^m\left(1+\frac{T}{X_i}\right).
\]

For \textup{(i)}, apply Lemma~\ref{lem:uniform-stationary-phase}
to the phase
$T^{-1}(\varphi(u)-(\mathsf Xk)\cdot u)$ for each
$k\in\mathcal B\cap\mathbb Z^m$.
The hypotheses are uniform in $k$, and injectivity gives at most
one stationary point in $U$.
The stationary point, when it exists, is
$u=\mathsf X^{-1}(x(k)-c)$, and
\[
 \nabla^2\varphi(u)=\mathsf X\nabla^2f(x(k))\mathsf X.
\]
Consequently the leading term, including $\ee(-k\cdot c)$, is
exactly the summand in \eqref{eq:general-B-formula}; the error per
frequency is $O_L((\det\mathsf X)T^{-m/2-1})$.
Every frequency with a stationary point on $\supp w$ lies in
$\mathcal B$ by its definition.
Summing these errors and the frequency tail proves
\eqref{eq:general-B-error}, since $T\ge1$ and $m\le3$.
The inverse function theorem and injectivity give a smooth inverse
on the open gradient image. The weighted dual summand is supported
on the image of the compact set $\supp w$, a compact subset of
this open image, so its zero extension is smooth.
Finally, the normalized inverse Hessian bound gives a leading term
of size $O_L((\det\mathsf X)T^{-m/2})$ per frequency.
Counting these frequencies and absorbing the error gives
\eqref{eq:general-B-bound}.

For \textup{(ii)}, small relative variation implies, throughout
$[-4,4]^m$,
\[
 \| (\nabla^2\varphi)^{-1}\|\ll_L(Tp)^{-1},\qquad
 |\det\nabla^2\varphi|\asymp |\det H_0|.
\]
If the phase $\varphi(u)-(\mathsf Xk)\cdot u$ has a stationary
point $u_*$ on $\supp\zeta$, its Hessian there is
$H_*=\nabla^2\varphi(u_*)$.
The identity
\[
 \begin{aligned}
 H_*^{-1}(\nabla^2\varphi(u)-H_*)
 &=[I+H_0^{-1}(H_*-H_0)]^{-1}\\
 &\quad\cdot\bigl[H_0^{-1}(\nabla^2\varphi(u)-H_0)
                    -H_0^{-1}(H_*-H_0)\bigr]
 \end{aligned}
\]
shows that its relative $C^9$ variation is $O_m(\varepsilon_m)$.
Taking $\varepsilon_m$ sufficiently small, Corollary~\ref{cor:absolute-stationary-phase}
therefore gives
\[
 |I_k|\ll_L(\det\mathsf X)|\det H_0|^{-1/2}.
\]

If there is no stationary point on $\supp\zeta$, the normalized
phase $T^{-1}(\varphi(u)-(\mathsf Xk)\cdot u)$ has inverse Hessian
$O_L(p^{-1})$ and derivatives of orders two through nine $O_L(1)$.
Lemma~\ref{lem:nonstationary-phase}, with eight integrations by
parts and the given $p$, therefore gives
\[
 |I_k|\ll_L(\det\mathsf X)(Tp^2)^{-8}.
\]
This argument requires no uniform lower bound for the gradient on
the amplitude support.
Summing the two estimates over $\mathcal B\cap\mathbb Z^m$ and
adding the frequency tail proves \eqref{eq:general-B-nearconstant};
the latter tail is absorbed because
$T^{-8}\ll_L(Tp^2)^{-8}$ when $p\ll_L1$.
\end{proof}

\bigskip

\noindent School of Mathematics and Statistics, Huazhong University of Science and Technology, Wuhan 430074, China.

\noindent\textit{Email address:} \texttt{tanbo@hust.edu.cn}

\bigskip

\noindent School of Mathematics and Statistics, Huazhong University of Science and Technology, Wuhan 430074, China.

\noindent\textit{Email address:} \texttt{zhangzhen\_@hust.edu.cn}
\bigskip

\noindent School of Mathematics and Statistics, Wuhan University of Technology, Wuhan 430070, China.

\noindent\textit{Email address:} \texttt{zhouql@whut.edu.cn}

\end{document}